\documentclass[12pt]{amsart}
\usepackage[numbers, square]{natbib}
\usepackage{amssymb}
\usepackage{amsmath}
\usepackage{amsfonts}
\usepackage{geometry}
\usepackage{bbm}
\usepackage{dsfont}
\usepackage{graphicx}
\usepackage{amsthm}
\usepackage{hyperref}
\usepackage[dvipsnames]{xcolor}
\usepackage{esint}

\providecommand{\U}[1]{\protect\rule{.1in}{.1in}}
\theoremstyle{plain}

\newtheorem{corollary}{Corollary}

\newtheorem{lemma}{Lemma}

\newtheorem{proposition}{Proposition}

\newtheorem{theorem}{Theorem}
\newtheorem{knlemma}{Lemma}

\theoremstyle{remark}
\newtheorem{remark}{Remark}
\numberwithin{equation}{section}

\DeclareMathOperator{\dist}{dist}

\newcommand{\be}{\beta}

\newcommand{\ww}{\widetilde{w}}

\newcommand{\iny}{\infty}

\newcommand{\LP}{\Delta}
\newcommand{\gr}{\nabla}

\newcommand{\norm}[1]{\left\vert\left\vert #1\right\vert\right\vert}

\newcommand{\set}[1]{\left\{#1\right\}}
\newcommand{\brac}[1]{\left[#1\right]}
\newcommand{\pr}[1]{\left( #1 \right) }

\begin{document}

\title[Doubling inequalities and propagation of smallness]
{Doubling inequalities and propagation of smallness for Schrödinger equations with singular potentials}
\author{Eugenia Malinnikova}
\address{
Department of Mathematics\\
Stanford University\\
Stanford, CA 94305, USA \\
Email: eugeniam@stanford.edu}
\author{ Jiuyi Zhu}
\address{
Department of Mathematics\\
Louisiana State University\\
Baton Rouge, LA 70803, USA\\
Email:  zhu@math.lsu.edu }
\subjclass[2020]{35J15, 35J10, 35B60, 35R30}\keywords {Propagation of smallness, Carleman estimates, Doubling inequalities}

\begin{abstract}
We study quantitative unique continuation for solutions
of the Schr\"o\-din\-ger equation with complex-valued singular potentials \(V\in L^t\), \(t>d/2\).  We obtain a
family of scale-invariant \(L^p\to L^q\) Carleman inequalities for the
Laplacian and use them to prove the
doubling inequalities with explicit dependence on the 
 \(L^t\)-norm of the potential. These estimates are combined with multiscale arguments to obtain propagation of smallness from arbitrary measurable sets of positive
measure. As an application of the main results, we prove an upper bound for the BMO norm for the logarithms of Dirichlet-Laplace eigenfunctions in simply connected planar Lipschitz domains.
\end{abstract}

\maketitle

\section{Introduction}
In this paper we study quantitative unique continuation for nontrivial
weak solutions of the Schr\"odinger equation
\begin{equation}
    \Delta u+V(x)u=0
    \label{ePDE}
\end{equation}
in a domain \(\Omega\subset\mathbb R^d\), \(d\ge2\). Both the solution
\(u\) and the potential \(V\) may be complex-valued.

There are two main approaches to unique continuation for second order elliptic
equations. The first approach is based on Carleman inequalities and  was
introduced in the original work by Carleman \cite{Car39}  to study
the two-dimensional Schr\"odinger equations with bounded potentials.
Jerison and Kenig \cite{JK85} obtained  the 
\(L^{2d/(d+2)}\to L^{2d/(d-2)}\) Carleman inequality with radial power weight and proved the strong unique continuation property for the Schr\"odinger equation
in dimensions \(d\ge3\) under the scale-critical assumption
\(V\in L^{d/2}_{\mathrm{loc}}\). 
 Carleman estimates with
linear exponential weights and the uniform Sobolev inequalities were developed by
Kenig, Ruiz, and Sogge \cite{KRS87}. Wolff \cite{W92} used linear
weights to show unique continuation properties of  Schr\"odinger equations with singular first-order terms and introduced the
measure-concentration lemma which allowed to select a weight adapted to the
solution. Koch and Tataru \cite{KT01} proved a version of the
 Carleman inequality and deduced the  strong unique continuation property for elliptic operators with
non-constant principal coefficients and singular lower order terms, combining the power weight approach of Jerison and Kenig with Wolff's lemma on weight adaptation.

The second approach to unique continuation is based on the Almgren frequency functions. 
Garofalo and Lin \cite{GL86,GL87} established almost monotonicity of
the frequency for divergence-form elliptic equations with Lipschitz
 coefficients and used it to obtain doubling inequalities and deduce the strong unique
continuation property for these equations. They also applied the frequency method to
the Schr\"odinger equations with homogeneous inverse-square potentials
\(
    V(x)=f(x/|x|)|x|^{-2},\)
     \(f\in L^\infty(S^{d-1}),
\)
a class of singular potentials not covered by the critical \(L^{d/2}\)-assumption.

Both methods have also been used to obtain quantitative forms of unique continuation,
including three-ball inequalities, doubling estimates, and bounds for
the order of vanishing. Using frequency-function methods, Kukavica
\cite{K98} obtained quantitative uniqueness estimates for
second-order elliptic equations with bounded lower-order terms. For the Schr\"odinger equations on  Riemannian manifolds, Bakri
\cite{B12} and Zhu \cite{Z16} independently obtained quantitative
vanishing estimates with sharp dependence on
\(\|V\|_{W^{1,\infty}}\). For singular lower-order terms, Davey and
Zhu \cite{DZ18,DZ19} developed \(L^p\to L^q\) Carleman
estimates and quantified the maximal vanishing order in terms of the
Lebesgue norms of the coefficients. Davey \cite{D20} subsequently obtained quantitative bounds for the vanishing order of solutions to the Schr\"odinger equation with $V\in L^t$, $t>d/2$, in dimensions $d\ge3$.

More recently, Dehman, Ervedoza, and Thabouti \cite{DET24}
proved global \(L^p\)-Carleman estimates for class of  general regular
 weights with some convexity condition, and obtained quantitative continuation estimates
for equations with singular zero and first order terms. Their method gives a sharper dependence on the norm of the potential for the fixed three domain inequality but does not imply the uniform doubling property.
 Caro, Ervedoza, and Thabouti
\cite{CET24} refined these estimates and combined them with Wolff's
 argument to treat first-order coefficients in the full
subcritical range.

Choulli and Takase \cite{CT25} used these estimates to obtain
fixed-ratio three-ball inequalities with explicit potential
dependence. They also proved a doubling inequality using a
singular radial power weight of Jerison and Kenig. However, the constant in their inequality depends on the small
radius and therefore does not give the uniform all-scale control, which we aim for in the present work.

The problem of quantitative continuation from sets of positive measure
goes back to a question of Landis: whether a three-ball inequality
holds when the smallest ball is replaced by an arbitrary
measurable set of positive measure. Nadirashvili
\cite{N86} obtained the first quantitative result in
this direction. Further 
estimates were proved by Vessella \cite{V20} using a different approach.  Regbaoui \cite{R01} also established  a qualitative unique
continuation result from sets of positive measure for solutions of 
Schr\"odinger type inequalities, $|\Delta u|\le A|u|+B|\nabla u|$.
For elliptic equations with singular lower-order terms, Malinnikova
and Vessella \cite{MV12} obtained the quantitative unique
continuation from  measurable sets of positive measure combining the Carleman inequality of Koch and Tataru with harmonic analysis results on Muckenhoupt weights. However, their result, as well as the previous results of Nadirashvili and Vessella for second order elliptic equations with bounded lower-order terms, gives an estimate weaker than the expected H\"older type estimate. 

The sharp analogue of the three-ball inequality with the
smallest ball replaced by an arbitrary measurable set of positive
measure, conjectured by Landis, was proved by Logunov and Malinnikova in \cite{LM18}. 
Their argument uses the almost monotonicity of the doubling
index, closely related to the monotonicity of the frequency function,
together with a multiscale subdivision procedure. More recently, the second author \cite{Zhu26}, and
Le Balc'h and Martin  \cite{LBM25}  independently extended this result to
equations with bounded lower-order terms.

In this paper we consider quantitative unique continuation for
\eqref{ePDE} with possibly complex-valued \(u\) and \(V\), assuming
that \(V\in L^t_{\mathrm{loc}}\) for some \(t>d/2\). Our main
analytic goal is a uniform doubling estimate: one prescribed
large-scale growth ratio should control the doubling index on every
smaller ball, uniformly in the center. 
For equations with bounded lower-order terms, such control is often
obtained from almost monotonicity of a frequency function. At present, to the best of our knowledge,
the frequency-function approach is less flexible for singular
potentials. Recent result of Davey \cite{D26}, 
assumes \(V\in L^t\) with \(t\ge d\). 

We prove a family of \(L^p\to L^q\) Carleman
inequalities for the Laplacian with the power weight
\(|x|^{-\tau}\). Related \(L^p\to L^q\) Carleman estimates were obtained in
\cite{DZ18,DZ19,D20}. The point of the estimate used here is that the
weight is exactly homogeneous and the dependence on the Carleman
parameter \(\tau\) is tracked throughout the full range of exponents.
 This allows us to  derive three ball inequality and then the following
 doubling inequalities.

\begin{theorem}
\label{th-1}
Let \(d\ge2\), \(t>d/2\). Suppose that \(u\in W^{1,2}_{\mathrm{loc}}(B_R\)  is a  nontrivial
weak solution
 of \eqref{ePDE}
in \(B_R\) with \(V\in L^t(B_R\).
If \(d\ge3\),  there exists  constant \(C=C(d,t)\) such that,
for every \(x\in B_{R/4}\) and every
\(0<r\le R/4\),
\begin{equation}
\|u\|_{L^2(B_{2r}(x))}
\le
\exp\left(C
\left(1+M_V(B_R)\right)^{\mu(d,t)}\right)\left(\frac{\|u\|_{L^2(B_R)}}
     {\|u\|_{L^2(B_{R/2})}}\right)^C
\|u\|_{L^2(B_r(x))},
\label{eq:main-doubling}
\end{equation}
where  $M_V(B_R)=R^{2-d/t}\|V\|_{L^t(B_R)}$ and \(\mu(d,t)\) is given in (\ref{mu3}).

If \(d=2\),  the same conclusion
holds for every \(\varepsilon>0\), with \(C=C(t,\varepsilon)\) and \(\mu(d,t)\) replaced by
\(
\mu_\varepsilon(2,t)\) defined in \eqref{mu2}.
\end{theorem}


This result allows us to adapt the multiscale argument of Logunov and Malinnikova \cite{LM18}, adding to the doubling index a scale-dependent potential term which decreases under subdivision, and obtain the following propagation of smallness estimate.
 \begin{theorem}
	   \label{th-2}
		Let $u$ be a solution of (\ref{ePDE}) in $B_R$ with $V\in L^t(B_R),$ $t>d/2$.
		Assume that
	$$|u|\le \delta\ {\text{ on}}\  E\subset B_{R/2}, \text{and} \  \ |u|\leq 1  {\text{ in}}\  B_R,    $$
		where $|E|>0$. Let  $K$ be a compact subset of $B_R$, then
		\begin{equation}
			\max_K|u|\le e^{C_0\left(1+M_V(B_R)\right)^{\mu(d,t)}
} \delta^{\alpha},
            \label{weak-st}
		\end{equation}
		where $C_0>0$ and $\alpha\in(0,1)$ depend on $R^{-d}|E|$ and $R^{-1}\dist(K,\partial B_R)$, and $\mu(d,t)$ is as in Theorem \ref{th-1}.
	\end{theorem}

The  paper is organized as follows.
In Section \ref{CarlEst}, we formulate the $L^p\to L^q$ Carleman inequality in Theorem \ref{Carlpq}. 
The proof of $L^p\to L^q$ Carleman inequality is postponed to Section \ref{s:Carleman}.
In Section \ref{Double}, we deduce the  three-ball  and doubling inequality of Theorem \ref{th-1}.
Theorem \ref{th-2} is obtained as a corollary of a more general Remez-type inequality in Section \ref{Remez}. 
We  give an application of our results to estimates of the BMO-norms of the logarithms of the Dirichlet-Laplace eigenfunctions in planar simply connected Lipschitz domains at the end of Section \ref{Remez}. 
Appendix \ref{s:local-estimates} includes basic regularity theory and oscillation estimates for solutions of the Schr\"odinger equations with complex-valued potentials following \cite{BK79}.

{ \bf Acknowledgment:} The project was started when the second author was visiting Stanford University in 2024. The second author thanks the Department of Mathematics at Stanford University for the warm hospitality and the wonderful academic atmosphere. The first author is supported in part by NSF grant DMS-2247185 and the second author is supported in part by  NSF grant  DMS-2453348 and Simons Foundation Collaboration Grant.

\section{Carleman estimates}
\label{CarlEst}
\subsection{Power weight Carleman estimate for the Laplace operator}
 Let us first recall the regularity of weak
solutions to \eqref{ePDE}. Assume that
\(
    u\in W^{1,2}_{\mathrm{loc}}(B_R)
\)
is a weak solution of the Schr\"odinger equation \eqref{ePDE}
in \(B_R\), where the complex valued \(V\in L^t(B_R)\), \(t>d/2\). 

By the standard Sobolev embedding arguments, see
Appendix~\ref{s:local-estimates} for details, there exist \(s>d/2\) and
\(\alpha>0\) such that
\[
u\in W^{2,s}_{\mathrm{loc}}(B_R)
\cap C^{0,\alpha}_{\mathrm{loc}}(B_R).
\]

In this section, we state the quantitative $L^p-L^{q}$ Carleman estimates for the Laplacian which will be used below. 
We use the notation
\[
\|f\|_{L^p(|x|^{-d}dx)}
    =
    \left(\int_{\mathbb R^d}|f(x)|^p |x|^{-d}\,dx\right)^{1/p}.
\]
The measure \(|x|^{-d}dx\) is invariant under dilation with respect to the origin, and for this reason the estimate below is well suited to scale-invariant quantitative
unique-continuation estimates.

\begin{theorem}
Let $d\ge3$ and $\frac{2d}{d+2} \leq p \le 2 \leq q \leq \frac{2d}{d-2}$. There
exists a constant $C$, depending on $d$, $p$, and $q$,  such that for any
$v\in C^{\infty}_{0}\pr{\mathbb R^d\setminus\set{0} }$ and every 
\(\tau>1\) satisfying
\(
    \operatorname{dist}(\tau,\mathbb N)>\frac13,
\) one has
\begin{align}
\tau^{\be} \| |x|^{-\tau} v\|_{L^q(|x|^{-d}dx)} 
\leq  C \| |x|^{-\tau+2}  \LP v\|_{L^p(|x|^{-d} dx)} ,
\label{mainCar}
\end{align}
where $\be = \be_d(p,q)=\frac{2+d}{4}+\frac{2-d}{2p}+\frac{d}{p}(\frac{1}{q}-\frac{1}{2})$.
 \label{Carlpq}

 For \(d=2\), assume
\(
    1<p\le2\le q\le \infty .
\) Then the above inequality holds with \(
    \beta=\beta_2(p,q)
    =
    1-\frac1p+\frac{2}{pq}.
\)
\end{theorem}

\begin{remark}
    In what follows we say that $\tau$ is non-resonant when \(
    \operatorname{dist}(\tau,\mathbb N)>\frac13,
\)
\end{remark}

The endpoint case, 
\(
    p=\frac{2d}{d+2},
  \ 
    q=\frac{2d}{d-2},
\) for $d\ge 3$,
is the classical  Carleman estimate of Jerison and Kenig \cite{JK85}. At this endpoint
the exponent above gives \(\beta=0\). The estimate in
Theorem~\ref{Carlpq} extends the endpoint estimate to a range of
\(p\)- and \(q\)-values and keeps track of the precise power of the Carleman parameter
\(\tau\).  This power is needed later when the potential term is absorbed and the
dependence on \(\|V\|_{L^t}\) is computed. The proof of Theorem \ref{Carlpq} is given in Section \ref{s:Carleman}.

\subsection{Application to the Schr\"odinger operators}
\label{CarlApp}
The proof of  Theorem \ref{Carlpq} is postponed to Section \ref{s:Carleman}. Now we show how the inequality is used to absorb the potential. We consider the Schr\"odinger operator $\Delta+V$ with $V\in L^t$, $d/2<t\le d$, and denote 
\begin{equation}\label{eq:mu}
\mu=\mu(d, t)=\frac{2td}{(2t-d)(d+2)},\quad {\text{for}}\ d\ge3, \qquad 
    \mu_\varepsilon(2, t)
    =
    \frac{t}{2(t-1)}+\varepsilon .
\end{equation}
\begin{proposition}
\label{CarlpqV}
\noindent\textup{(i)}
Let \(d\ge 3\) and
\(
    \frac d2<t\le d.
\)
Set
\[
    p=\frac{2d}{d+2},
    \qquad
    \beta=\frac{d+2}{2d}.
\]
There exists a constant \(C=C(d,t)\) such that, for every
\(
    V\in L^t(B_R),\)
\(    v\in C^\infty_0(B_R\setminus\{0\}),
\)
and every non-resonant \(\tau>1\) satisfying
\[
    \tau
    \ge
    C
    \left(
        1+R^{2-d/t}\|V\|_{L^t(B_R)}
    \right)^\mu,
\]
one has
\begin{equation}
    \tau^\beta
    \bigl\||x|^{-\tau}v\bigr\|_{L^2(|x|^{-d}\,dx)}
    \le
    C
    \bigl\|
        |x|^{-\tau+2}(\Delta+V)v
    \bigr\|_{L^p(|x|^{-d}\,dx)} .
    \label{main1}
\end{equation}

\noindent\textup{(ii)}
Let \(d=2\), 
\(
    1<t\le 2,
\)
and fix \(\varepsilon>0\). 
There exist
\(
    p_\varepsilon=p_\varepsilon(t,\varepsilon)\in(1,t)
\)
and a constant \(C_\varepsilon=C_\varepsilon(t,\varepsilon)\ge 1\)
such that, for every
\(    V\in L^t(B_R),\) and \(
    v\in C^\infty_0(B_R\setminus\{0\}),
\)
and every non-resonant \(\tau>1\) satisfying
\[
    \tau
    \ge
    C_\varepsilon
    \left(
        1+R^{2-2/t}\|V\|_{L^t(B_R)}
    \right)^{\mu_\varepsilon(2,t)},
\]
the inequality \eqref{main1} holds with  $p=p_\varepsilon$, $\beta=1$, and $C=C_\varepsilon$.
\end{proposition}

\begin{proof}
(i) Since we choose $p=\frac{2d}{d+2}$, then  $\be=\be_d(\frac{2d}{d+2},2)=\frac {d+2} {2d}$. Hence we define $q$ such that $\frac 1 p - \frac 1 q = \frac 1 t$. 
Since $ \frac{d}{2}<t\le d$, we get $q \in [2, \frac{2d}{d-2}).$  For this choice of $p$ and $q$, we have 
\begin{align*}
    \beta_*=\beta_d(p,q)=\frac{(d+2)(2t-d)}{2dt}.
\end{align*}
Applying Theorem \ref{Carlpq} with  $q = \frac{2dt}{dt + 2t - 2d }$ and  also with $q = 2$, and adding the inequalities,  we get
\begin{align*}
\mathcal{J}:&=\tau^{\be}\norm{ |x|^{-\tau}v}_{L^2(|x|^{-d}dx)}
+ \tau^{\be_*} \norm{ |x|^{-\tau} v}_{L^q(|x|^{-d}dx)} 
\le C \norm{ |x|^{-\tau+2} \LP v}_{L^p(|x|^{-d} dx)} \\
&\le C \norm{ |x|^{-\tau+2} \pr{\LP v + Vv}}_{L^p(|x|^{-d} dx)}
+ C \norm{ |x|^{-\tau+2} V v}_{L^p(|x|^{-d} dx)} \\
&\le C \norm{ |x|^{-\tau+2} \pr{\LP v + Vv}}_{L^p(|x|^{-d} dx)} 
+ C \norm{ |x|^{2- \frac d p + \frac d q}}_{L^\iny} \norm{V}_{L^t(B_R)} \norm{ |x|^{-\tau}v}_{L^q(|x|^{-d}dx)} ,
\end{align*}
where we used the triangle inequality and H\"older's inequality.
Since $\frac 1 p - \frac 1 q = \frac{1}{t}$ and $t>\frac d 2$, we have $2- \frac d p + \frac d q = 2 - \frac d t > 0$. Therefore
\[
\mathcal{J}
\le C \norm{ |x|^{-\tau+2} \pr{\LP v + Vv}}_{L^p(|x|^{-d} dx)} 
+ C R^{2-d/t}\|V\|_{L^t(B_R)} \norm{ |x|^{-\tau} v}_{L^q(|x|^{-d}dx)} .
\]
By choosing $\tau>1$ such that $\tau^{\be_*} > CR^{2-d/t}\|V\|_{L^t(B_R)}$, we can incorporate the last term on the right hand side of the last inequality into the left hand side.
Note that $\mu = \frac 1 {\be_*}$.  Thus, the following Carleman estimates hold,
\begin{align*}
\tau^{\be} \norm{ |x|^{-\tau} v}_{L^2(|x|^{-d}dx)}
\le C \norm{ |x|^{-\tau+2} \pr{\LP v + Vv}}_{L^p(|x|^{-d} dx)}
\end{align*}
for non-resonant  $\tau> 1+C(R^{2-d/t}\|V\|_{L^t(B_R)})^\mu$.

(ii) The proof is similar, we choose 
\(p_\varepsilon>1\) sufficiently close to \(1\), and define \(q_\varepsilon\) by
\(
    \frac1{p_\varepsilon}-\frac1{q_\varepsilon}=\frac1t .
\)
Then \(q_\varepsilon\ge2\). Write
\[
    \beta_2(p_\varepsilon, q_\varepsilon)
    =
    1-\frac1{p_\varepsilon}
    +
    \frac{2}{p_\varepsilon q_\varepsilon}.
\]
We choose \(p_\varepsilon\) so close to \(1\) that
\[
  { \frac1 {\beta_2(p_\varepsilon, q_\varepsilon)}
    \le
    \frac{t}{2(t-1)}+\varepsilon=\mu_\varepsilon(2, t).}
\]
Applying Theorem~\ref{Carlpq} with \((p,q)=(p_\varepsilon,2)\) and
\((p,q)=(p_\varepsilon,q_\varepsilon)\), the same absorption argument as above gives
the two-dimensional statement with constants depending on \(t\) and \(\varepsilon\).
\end{proof}

For $d < t\leq \infty$, we apply the $L^2$ type Carleman estimates obtained in \cite{Z15}. We fix a large enough dimensional constant \(\Lambda\) and set
\[
    \ell_R(x)=\log\frac{|x|}{\Lambda R}.
\]
Thus \(\ell_R<0\) in \(B_R\). We also define \(\varphi(x)=\varphi_R(x)=-\ell_R-2\log|\ell_R|.
\)
Equivalently,
\[
    e^{\tau\varphi(x)}
    =
    \left(\frac{\Lambda R}{|x|}\right)^\tau
    |\ell_R|^{-2\tau}.
\]

   The following Carleman estimate was obtained in Corollary 2 in \cite{Z15}:
\begin{knlemma}
There exist constants \(C_1>0\) and \(\tau_0>0\), depending only on
\(d\), such that for every \(R>0\), every \(0<\rho<R\), every
\(v\in C^\infty_c(B_R\setminus B_\rho)\), and every
\(\tau\ge\tau_0\), one has
\begin{align}
\bigl\||x|^2 e^{\tau\varphi}\Delta v\bigr\|_{L^2(|x|^{-d}\,dx)}
&\ge C_1\tau^{3/2}
\bigl\|e^{\tau\varphi}|\ell_R|^{-1}v\bigr\|_{L^2(|x|^{-d}\,dx)}+ C_1\tau^{1/2}
\bigl\|e^{\tau\varphi}|\ell_R|^{-1}|x|\nabla v
\bigr\|_{L^2(|x|^{-d}\,dx)}
 \notag \\
 &+ C_1\tau^{1/2}\rho^{1/2}
\bigl\|e^{\tau\varphi}|x|^{-1/2}v
\bigr\|_{L^2(|x|^{-d}\,dx)} .
\label{car-est}
\end{align}
\end{knlemma}

If we consider $V(x)\in L^t(B_R)$ for $t \in [d, \infty]$, we have the following quantitative Carleman estimates.

\begin{proposition}
\label{lem:large-t-carleman}
Let \(d\ge 3\), \(d<t<\infty\), and \(0<\rho<R\).  There exist constants
\(C_0,C_1>0\), depending only on \(d\) and \(t\), such that for every \(v\in C^\infty_0(B_R\setminus B_\rho)\), if
\[
\tau\ge
C_0\left(
1+R^{2-d/t}\|V\|_{L^t(B_R)}
\right)^{\frac{2t}{3t-2d}},
\]
then
\begin{align}
\bigl\||x|^2e^{\tau\varphi}(\Delta &+V)v
\bigr\|_{L^2(|x|^{-d}\,dx)}
\ge
C_1\tau^{3/2}
\bigl\|e^{\tau\varphi}|\ell_R|^{-1}v
\bigr\|_{L^2(|x|^{-d}\,dx)}
\notag \\
&+
C_1\tau^{1/2}
\bigl\|e^{\tau\varphi}|\ell_R|^{-1}|x|\nabla v
\bigr\|_{L^2(|x|^{-d}\,dx)}+
C_1\tau^{1/2}\rho^{1/2}
\bigl\|e^{\tau\varphi}|x|^{-1/2}v
\bigr\|_{L^2(|x|^{-d}\,dx)}.
\label{Carl:ln}
\end{align}

If \(d=2\) and \(2<t<\infty\), then for every \(\varepsilon>0\)
there exist constants \(C_{0,\varepsilon},C_{1,\varepsilon}>0\),
depending only on \(t\) and \(\varepsilon\), such that
\eqref{Carl:ln} holds, with \(C_1\) replaced by
\(C_{1,\varepsilon}\),  whenever
\[
\tau\ge
C_{0,\varepsilon}\left(
1+R^{2-2/t}\|V\|_{L^t(B_R)}
\right)^{\frac{2t}{3t-4}+\varepsilon}.
\]

If \(d\ge2\) and \(t=\infty\), there exist constants \(C_0,C_1>0\),
depending only on \(d\), such that \eqref{Carl:ln} holds whenever
\[
\tau\ge
C_0\left(
1+R^2\|V\|_{L^\infty(B_R)}
\right)^{2/3}.
\]
\label{proposi-2}
\end{proposition}

\begin{proof}
By increasing the constant in the lower bound for \(\tau\), we may
assume that \(\tau\ge\max\{\tau_0,1\}\), where \(\tau_0\) is the
constant in Lemma A. Set
\begin{align*}
A_0:=
\bigl\|e^{\tau\varphi}|\ell_R|^{-1}v
\bigr\|_{L^2(|x|^{-d} dx)},\quad 
A_1:=
\bigl\|e^{\tau\varphi}|\ell_R|^{-1}|x|\nabla v
\bigr\|_{L^2(|x|^{-d} dx)}, \\
A_2:=
\bigl\|e^{\tau\varphi}|x|^{-1/2}v
\bigr\|_{L^2(|x|^{-d} dx)},
 \quad X:=\tau^{3/2}A_0+\tau^{1/2}A_1+
\tau^{1/2}\rho^{1/2}A_2 
\end{align*}
By Lemma A and the triangle inequality,
there is a constant \(c_*>0\), depending only on \(d\), such that
\begin{align}
\bigl\||x|^2e^{\tau\varphi}(\Delta+V)v
\bigr\|_{L^2(|x|^{-d}\,dx)}
&\ge
\bigl\||x|^2e^{\tau\varphi}\Delta v
\bigr\|_{L^2(|x|^{-d}\,dx)}
-
\bigl\||x|^2e^{\tau\varphi}Vv
\bigr\|_{L^2(|x|^{-d}\,dx)}
\notag \\
&\ge
c_*X-
\bigl\||x|^2e^{\tau\varphi}Vv
\bigr\|_{L^2(|x|^{-d}\,dx)}.
\label{absorb-1}
\end{align}
It is therefore enough to show that the last term in (\ref{absorb-1}) is at most
\(c_*X/2\).

Assume first that \(d\ge3\) and \(d<t<\infty\). Let \(\lambda>0\)
and decompose \(V=V_1+V_2\), where
\(V_1=V\mathbf 1_{\{|V|>\lambda\}}\) and
\(V_2=V\mathbf 1_{\{|V|\le\lambda\}}\). Then
\begin{align}
\|V_2\|_{L^\infty(B_R)}\le\lambda,
\qquad \
\|V_1\|_{L^d(B_R)}
\le
\lambda^{1-t/d}\|V\|_{L^t(B_R)}^{t/d}.
\label{two-in}
\end{align}

Since \(|x|^2|\ell_R(x)|\le CR^2\) in \(B_R\), the bounded part in (\ref{two-in})
satisfies
\begin{align}
\bigl\||x|^2e^{\tau\varphi}V_2v
\bigr\|_{L^2(|x|^{-d}\,dx)}
\le C\lambda R^2A_0.
\label{bb-bb}
\end{align}
Thus this term is absorbed by the first term of \(X\) provided
\(\lambda R^2\le c\tau^{3/2}\).

For the singular part in (\ref{two-in}), define
\(F=|x|^{2-d/2}e^{\tau\varphi}v\). Applying Hölder's inequality and the Sobolev
inequality give
\begin{multline}
\bigl\||x|^2e^{\tau\varphi}V_1v
\bigr\|_{L^2(|x|^{-d}\,dx)}
=
\|V_1F\|_{L^2(B_R)}
\\\le
\|V_1\|_{L^d(B_R)}
\|F\|_{L^{2d/(d-2)}(B_R)}
\le
C\|V_1\|_{L^d(B_R)}
\|\nabla F\|_{L^2(B_R)}.
\label{holder-so}
\end{multline}
Since \(|\nabla\varphi(x)|\le C|x|^{-1}\),
\(|x||\ell_R(x)|\le CR\), and \(\tau\ge1\), we have
\[
\|\nabla F\|_{L^2(B_R)}
\le
C\tau
\bigl\||x|^{1-d/2}e^{\tau\varphi}v
\bigr\|_{L^2(B_R)}
+
C
\bigl\||x|^{2-d/2}e^{\tau\varphi}\nabla v
\bigr\|_{L^2(B_R)}
\le
CR\bigl(\tau A_0+A_1\bigr).
\]
Consequently, the last inequality and (\ref{holder-so} imply that
\[
\bigl\||x|^2e^{\tau\phi}V_1v
\bigr\|_{L^2(|x|^{-d}\,dx)}
\le
CR\|V_1\|_{L^d(B_R)}
\bigl(\tau A_0+A_1\bigr).
\]
Hence this term is absorbed by the first two terms of \(X\) provided
\(R\|V_1\|_{L^d(B_R)}\le c\tau^{1/2}\).

Choose \(\lambda=cR^{-2}\tau^{3/2}\), where \(c>0\) is sufficiently
small. The bounded part in (\ref{bb-bb}) is then absorbed. From (\ref{two-in}), The 
singular part can be absorbed if
\[
R\lambda^{1-t/d}\|V\|_{L^t(B_R)}^{t/d}
\le c\tau^{1/2}.
\]
After substituting the value of \(\lambda\), this becomes
\[
\left(
R^{2-d/t}\|V\|_{L^t(B_R)}
\right)^{t/d}
\le
c\tau^{(3t-2d)/(2d)}.
\]
Hence both terms are absorbed whenever
\[
\tau\ge
C\left(
1+R^{2-d/t}\|V\|_{L^t(B_R)}
\right)^{\frac{2t}{3t-2d}}.
\]
Thus the Carleman estimates (\ref{Carl:ln}) follows from (\ref{absorb-1})
for \(d\ge3\) and \(d<t<\infty\).

Suppose now that \(d=2\) and \(2<t<\infty\). Choose
\(a_\varepsilon\in(2,t)\) sufficiently close to \(2\) so that
\begin{align}
\frac{2t}{3t-2a_\varepsilon}
\le
\frac{2t}{3t-4}+\varepsilon.
\label{aaa-e}
\end{align}
Using the same decomposition \(V=V_1+V_2\). we have
\begin{align}
\|V_2\|_{L^\infty(B_R)}\le\lambda,
\qquad
\|V_1\|_{L^{a_\varepsilon}(B_R)}
\le
\lambda^{1-t/a_\varepsilon}
\|V\|_{L^t(B_R)}^{t/a_\varepsilon}.
\label{bb-bb-1}
\end{align}
Arguing as in (\ref{bb-bb}), the \(V_2\)-term in the last inequality is absorbed as above if
\(\lambda R^2\le c\tau^{3/2}\).

Let \(q_\varepsilon=2a_\varepsilon/(a_\varepsilon-2)\) and 
\(F=|x|e^{\tau\varphi}v\). Hölder's inequality and the
two-dimensional Sobolev inequality give
\begin{align*}
\bigl\||x|^2e^{\tau\varphi}V_1v
\bigr\|_{L^2(|x|^{-2}\,dx)}
&=
\|V_1F\|_{L^2(B_R)}
\le
\|V_1\|_{L^{a_\varepsilon}(B_R)}
\|F\|_{L^{q_\varepsilon}(B_R)}
\\
&\le
C_\varepsilon
R^{(a_\varepsilon-2)/a_\varepsilon}
\|V_1\|_{L^{a_\varepsilon}(B_R)}
\|\nabla F\|_{L^2(B_R)}.
\end{align*}
The bounds on \(\varphi\) and \(\ell_R\) imply
\[
\|\nabla F\|_{L^2(B_R)}
\le
CR\bigl(\tau A_0+A_1\bigr).
\]
Therefore
\[
\bigl\||x|^2e^{\tau\varphi}V_1v
\bigr\|_{L^2(|x|^{-2}\,dx)}
\le
C_\varepsilon
R^{2-2/a_\varepsilon}
\|V_1\|_{L^{a_\varepsilon}(B_R)}
\bigl(\tau A_0+A_1\bigr).
\]
This term is absorbed if
\[
R^{2-2/a_\varepsilon}
\|V_1\|_{L^{a_\varepsilon}(B_R)}
\le
c_\varepsilon\tau^{1/2}.
\]
Choose \(\lambda=cR^{-2}\tau^{3/2}\) as in the last case, the $V_1$ term in (\ref{bb-bb-1}) can be absorbed if 
\[
\left(
R^{2-2/t}\|V\|_{L^t(B_R)}
\right)^{t/a_\varepsilon}
\le
c_\varepsilon
\tau^{(3t-2a_\varepsilon)/(2a_\varepsilon)}.
\]
Therefore it is  enough that
\[
\tau\ge
C_\varepsilon
\left(
1+R^{2-2/t}\|V\|_{L^t(B_R)}
\right)^{\frac{2t}{3t-2a_\varepsilon}}.
\]
Then (\ref{aaa-e}) gives the stated lower bound for
\(\tau\) in this case and Carleman estimate (\ref{Carl:ln}) follows from (\ref{absorb-1}).

Finally, suppose that \(t=\infty\). No decomposition is needed.
Indeed,
\begin{align}
\bigl\||x|^2e^{\tau\varphi}Vv
\bigr\|_{L^2(|x|^{-d}\,dx)}
\le
\|V\|_{L^\infty(B_R)}
\bigl\||x|^2e^{\tau\varphi}v
\bigr\|_{L^2(|x|^{-d}\,dx)}
\le
CR^2\|V\|_{L^\infty(B_R)}A_0.
\label{potential-1}
\end{align}
This term is absorbed by the first term of \(X\) whenever
\(R^2\|V\|_{L^\infty(B_R)}\le c\tau^{3/2}\), which follows from
\[
\tau\ge
C\left(
1+R^2\|V\|_{L^\infty(B_R)}
\right)^{2/3}.
\]
After the potential term $V(x)$ is absorbed in (\ref{potential-1}), \eqref{Carl:ln} follows with
a  positive constant \(C_1\).
\end{proof}

Combining Proposition \ref{CarlpqV} and Proposition \ref{lem:large-t-carleman}, we extend the definition of \(\mu=\mu(d,t)\) by setting
\begin{equation}\label{mu3}
    \mu(d,t)
    =
    \begin{cases}
    \dfrac{2dt}{(2t-d)(d+2)}, & d/2<t\le d,\\
    \min\left\{
        \dfrac{2d}{d+2},
        \dfrac{2t}{3t-2d}
    \right\}, & d<t\le\infty .
    \end{cases}
\end{equation}
Here \(\frac{2t}{3t-2d}\) is understood as \(\frac23\) when \(t=\infty\).

If \(d=2\), fix \(\varepsilon>0\). For \(t>1\), set
\begin{equation}\label{mu2}
    \mu_\varepsilon(2,t)
    =
    \begin{cases}
    \dfrac{t}{2(t-1)}+\varepsilon, & 1<t\le2,\\
    \min\left\{
        1+\varepsilon,
        \dfrac{2t}{3t-4}+\varepsilon
    \right\}, & 2<t<\infty,\\
    \dfrac23, & t=\infty .
    \end{cases}
\end{equation}
Then, in all estimates below, the scale-normalized potential contribution may be written
as
\[
    \left(
        1+
        R^{2-d/t}\|V\|_{L^t(B_R)}
    \right)^{\mu(d,t)}
\ \mbox{for} \ d\ge3, \quad{\text{and}}\quad 
    \left(
        1+
        R^{2-2/t}\|V\|_{L^t(B_R)}
    \right)^{\mu_\varepsilon(2,t)}
\ \mbox{for} \  d=2.\]
In the case \(t=\infty\), the factor
\(R^{2-d/t}\|V\|_{L^t(B_R)}\) is understood as
\(R^2\|V\|_{L^\infty(B_R)}\).

\begin{remark}
The Carleman estimates above are stated for smooth compactly
supported functions, but we shall also apply them to functions of the
form \(v=\chi u\), where \(u\) is a weak solution of \eqref{ePDE}
and \(\chi\) is a smooth cutoff whose support stays away from the
center of the weight. This follows by a standard approximation
argument.
\end{remark}

\section{The proof of doubling inequality}
\label{Double}
\subsection{Three-ball inequality}
We will use a quantitative Caccioppoli inequality for the second order elliptic equation (\ref{ePDE}) with singular lower-order terms, which was obtained in \cite{DZ19}.
\begin{knlemma}
Let $t > \frac d 2$,  $\norm{V}_{L^t\pr{B_{R}}} \le M$, and 
let $u$ be a solution to \eqref{ePDE} in $B_R$.
Then there exists a constant $C$, depending only on $d$ and $t$, such that
\begin{equation}
\|\nabla u\|^2_{L^2(B_r)}\leq
C\brac{\frac{1}{(R-r)^2}+M^{\frac{2t}{2t-d}}}\|
u\|^2_{L^2( B_R)}
\label{Cacc-1}
\end{equation}
for any $r<R$.
\label{CaccLem}
\end{knlemma}

By the local boundedness estimate, see Lemma \ref{lem:L2-Linfty-local} in Appendix, if
\(u\) solves (\ref{ePDE}) in \(B_R(x_0)\), then for $\theta<1$,
\[
    \|u\|_{L^\infty(B_{\theta R})}
    \le
    C_\theta R^{-d/2}
    \left(
        1+
        \left(
            R^{2-d/t}\|V\|_{L^t(B_R)}
        \right)^{\frac{2t}{2t-d}}
    \right)^{d/4}
    \|u\|_{L^2(B_R)}.
\]

Our first aim is to obtain three-ball inequality for solutions of (\ref{ePDE}) when $V(x)\in L^t$ for $\frac{d}{2}<t\leq \infty$. The deduction of three-ball theorem from Carleman inequalities is rather standard, see e.g. \cite{Z15}, we sketch the proof omitting some details. We will need the following notation
\begin{align}A_V(B_{r_2})
    =
    \left(
        1+
        r_2^{2-d/t}\|V\|_{L^t(B_{r_2})}
    \right)^\mu,
    \label{AAA-1}
\end{align}
where $\mu$ is defined by \eqref{mu3} for $d\ge 3$. For $d=2$, we fix $\varepsilon>0$ and define $\mu$ by \eqref{mu2}.
\begin{lemma}
\label{lem:three-ball}
Let \(d\ge2\), \(t>d/2\), and let \(u\) be a weak solution of \eqref{ePDE}
in \(B_{r_2}\). Let
$0<r_1<r<r_2.$
Then there exist constants \(C>0\) and \(\alpha\in(0,1)\), depending only on
\(d,t\) the ratios \(r/r_1\) and \(r_2/r\), and $\varepsilon$ when $d=2$ such that
\begin{equation}\label{eq:three-ball}
    \|u\|_{L^2(B_r)}
    \le
    C e^{CA_V(B_{r_2})}
    \|u\|_{L^2(B_{r_1})}^{\alpha}
    \|u\|_{L^2(B_{r_2})}^{1-\alpha}.
\end{equation}
\end{lemma}

\begin{proof}
By scaling, it is enough to consider the case \(r_2=1\). We write
\(A=A_V(B_1)\) and choose
\[
\rho=\frac{1+r}{2},\qquad s=\frac{1+\rho}{2}.
\]
Let \(\chi\in C^\infty_0(B_s)\) satisfy
\(
\chi=0\) in \(B_{r_1/2},\) and
\(\chi=1\) in \(B_\rho\setminus B_{3r_1/4}.
\)
We may choose \(\chi\) so that
\[
|\nabla\chi|\le Cr_1^{-1},\qquad
|\Delta\chi|\le Cr_1^{-2}
\]
on \(B_{3r_1/4}\setminus B_{r_1/2}\), while
\(|\nabla\chi|+|\Delta\chi|\le C\) on \(B_s\setminus B_\rho\).
Here and below the constants may depend on the ratios \(r/r_1\)
and \(1/r\).

Since \(u\) is a solution to \eqref{ePDE}, we have
\begin{equation}
(\Delta+V)(\chi u)
=2\nabla\chi\cdot\nabla u+(\Delta\chi)u.
\label{sim-eqn}
\end{equation}
The Caccioppoli inequality in (\ref{Cacc-1}) gives
\begin{align}
r_1\|\nabla u\|_{L^2(B_{3r_1/4})}
\le Ce^{CA}\|u\|_{L^2(B_{r_1})},\qquad
\|\nabla u\|_{L^2(B_s)}
\le Ce^{CA}\|u\|_{L^2(B_{1})}.
\label{caccio-1}
\end{align}
Indeed, all polynomial factors in (\ref{Cacc-1}) in the scale-normalized
\(L^t\)-norm of \(V\) are bounded by \(Ce^{CA}\).

We use the potential-absorbed Carleman estimate corresponding to the
exponent in the definition of \(\mu(d,t)\), or of
\(\mu_\varepsilon(2,t)\) when \(d=2\). If \(d\ge3\) and
\(d/2<t\le d\), or if \(d=2\) and \(1<t\le2\), we apply
Proposition~\ref{CarlpqV} with the power weight. If \(d\ge3\) and
\(t>d\), the two alternatives in the definition of \(\mu(d,t)\)
come respectively from Proposition~\ref{CarlpqV}, applied with
exponent \(d\), and from Proposition~\ref{proposi-2}; we use the
estimate which gives the smaller exponent. Similarly, if \(d=2\)
and \(t>2\), we use either Proposition~\ref{CarlpqV}, applied with
exponent \(2\), or Proposition~\ref{proposi-2}, according to which
alternative realizes \(\mu_\varepsilon(2,t)\). For \(t=\infty\), we
use Proposition~\ref{proposi-2}. In what follows, \(\mu\) denotes the
corresponding exponent.


For every non-resonant \(\tau\ge C_0A\), applying Proposition~\ref{CarlpqV} with $v=\chi u$ and 
H\"older's inequality, by (\ref{sim-eqn}),  we have
\begin{align*}
\tau^\beta \| |x|^{-\tau}\chi u\|_{L^2(|x|^{-d}dx)} &\leq C\|  |x|^{-\tau+2}(\triangle (\chi u)+V(x)\chi u)\|_{L^p(|x|^{-d}dx)}  
\\&\leq C \| |x|^{-\tau+2}[\triangle\chi u+2 \nabla \chi \cdot \nabla u ] \|_{L^p(|x|^{-d}dx)},
\end{align*}
By the properties of cut-off function $\chi$ and the estimates (\ref{caccio-1}), we obtain 
\begin{align}
\|u\|_{L^2(B_r\setminus B_{r_1})}
\le Ce^{CA}\left[
\left(\frac{2r}{r_1}\right)^\tau \|u\|_{L^2(B_{r_1})}
+\left(\frac r\rho\right)^\tau \|u\|_{L^2(B_{1})}
\right].
\label{appl-car-1}
\end{align}


When applying Proposition~\ref{proposi-2}, we set \(\rho=r_1/2\) in \(B_1\). Recall that
\(\varphi=\varphi_1\) is decreasing as a function of \(|x|\) in
\(B_1\). The factors involving \(|\ell_1|\) are bounded above and
below on each of the annuli under consideration by constants
depending only on \(r/r_1\) and \(1/r\). Hence, for every
\(\tau\ge C_0A\), the similar arguments as (\ref{appl-car-1}) shows that
\begin{align}
\|u\|_{L^2(B_r\setminus B_{r_1})}
\le Ce^{CA}\left(
e^{\tau(\varphi(r_1/2)-\varphi(r))}\|u\|_{L^2(B_{r_1})}
+e^{-\tau(\varphi(r)-\varphi(\rho))}\|u\|_{L^2(B_{1})}\right).
\label{appl-car-2}
\end{align}  

Adding $\|u\|_{L^2(B_{r_1})}$ to the left hand side of (\ref{appl-car-1}) or (\ref{appl-car-2}) , we conclude in either case
that there exist \(a,b>0\), depending only on \(r/r_1\) and \(1/r\),
such that
\begin{align}
\|u\|_{L^2(B_r)}\le Ce^{CA}\left(e^{a\tau}\|u\|_{L^2(B_{r_1})}+e^{-b\tau}\|u\|_{L^2(B_1)}\right)
\label{abab-1}
\end{align}
for every admissible \(\tau\ge C_0A\). In the power-weight case  (\ref{appl-car-1})  we
may take
\[
a=\log\frac{2r}{r_1},\qquad b=\log\frac{\rho}{r},
\]
while in the logarithmic-weight case (\ref{appl-car-2}) we take
\[
a=\varphi(r_1/2)-\varphi(r),\qquad
b=\varphi(r)-\varphi(\rho).
\]

 Since \(\|u\|_{L^2(B_{r_1})}\le \|u\|_{L^2(B_1)}\), the number
\begin{align}
\tau_*:=\frac{1}{a+b}\log\frac {\|u\|_{L^2(B_1)}}{\|u\|_{L^2(B_{r_1})}}
\end{align}
is nonnegative. Set
\(
\alpha=\frac{b}{a+b}.
\)
Then
\begin{align}
e^{a\tau_*}\|u\|_{L^2(B_{r_1})}=e^{-b\tau_*}\|u\|_{L^2(B_{1})}=\|u\|_{L^2(B_{r_1})}^\alpha \|u\|_{L^2(B_1)}^{1-\alpha}.
\label{tau-star}
\end{align}

In the logarithmic-weight case (\ref{appl-car-2}), we choose
\(\tau=\max\{C_0A,\tau_*\}\). In the power-weight case (\ref{appl-car-1}), we choose a
non-resonant \(\tau\) such that
\[
\max\{C_0A,\tau_*\}
\le\tau\le
\max\{C_0A,\tau_*\}+1.
\]
In both cases, for \(\tau\ge\tau_*\), it follows from (\ref{tau-star}) that
\begin{align}
e^{-b\tau}\|u\|_{L^2(B_1)}\le \|u\|_{L^2(B_{r_1})}^\alpha \|u\|_{L^2(B_1)}^{1-\alpha}.
\label{coro-1}
\end{align}
Moreover, since \(A\ge1\),
\begin{align}
e^{a\tau}\|u\|_{L^2(B_{r_1})}
=e^{a(\tau-\tau_*)}\|u\|_{L^2(B_{r_1})}^\alpha \|u\|_{L^2(B_1)}^{1-\alpha}
\le e^{CA}\|u\|_{L^2(B_{r_1})}^\alpha \|u\|_{L^2(B_1)}^{1-\alpha}.
\label{coro-2}
\end{align}
It follows from (\ref{abab-1}), (\ref{coro-1}) and  (\ref{coro-2}) that
\[
\|u\|_{L^2(B_r)}\le Ce^{CA}\|u\|_{L^2(B_{r_1})}^\alpha \|u\|_{L^2(B_1)}^{1-\alpha}.
\]
Rescaling back to \(B_{r_2}\) proves \eqref{eq:three-ball}.
\end{proof}

By repeatedly  using the three-ball inequality \eqref{eq:three-ball} and propagation of smallness argument, we can show the following results. 

\begin{lemma}
\label{lem:change-center}
Let \(u\) be a nontrivial weak solution of \eqref{ePDE}
in \(B_R\), where \(V\in L^t(B_R)\), \(t>d/2\). Let
\(
    0<\rho<\frac{R}{16}\) and
    \(x,y\in B_{R/2}.
\)
Set as above
\[
    A=
    \left(
        1+
        R^{2-d/t}\|V\|_{L^t(B_R)}
    \right)^\mu .
\]
Then there exist constants \(C>0\) and \(\gamma\in(0,1)\), depending only on
\(d,t\) and on \(\rho/R\), such that
\begin{equation}
    \|u\|_{L^2(B_\rho(y))}
    \le
    C e^{CA}
    \|u\|_{L^2(B_\rho(x))}^{\gamma}
    \|u\|_{L^2(B_R)}^{1-\gamma}.
    \label{change:center}
\end{equation}
Consequently, if
\[
    N_R=
    \log
    \frac{\|u\|_{L^2(B_R)}}
         {\|u\|_{L^2(B_{R/2})}},
\]
then for every \(x\in B_{R/2}\),
\begin{equation}
    \|u\|_{L^2(B_\rho(x))}
    \ge
    e^{-C(A+N_R)}
    \|u\|_{L^2(B_{R/2})},
    \label{lower:center}
\end{equation}
where \(C\) depends only on \(d,t\), on \(\rho/R\), and on $\varepsilon$ for $d=2$.
\end{lemma}

\begin{proof}
First we  prove \eqref{change:center}. Choose a chain of points
\(
    x=x_0',x_1',\ldots,x_m'=y
\)
such that
\[
    |x_{j+1}'-x_j'|\le \frac{\rho}{2},
    \qquad
    m\le C\frac{R}{\rho},
\]
and
\(
    B_{4\rho}(x_j')\subset B_R
\)
for all \(j\). The constant in this construction depends only on \(d\) and on
\(\rho/R\).
We apply Lemma~\ref{lem:three-ball} with radii
\(
    \frac{\rho}{2}<\rho<4\rho
\)
and center \(x_{j+1}'\). Since \(B_{\rho/2}(x_{j+1}')\subset B_\rho(x_{j}')\), we get
\[
    \|u\|_{L^2(B_\rho(x_{j+1}'))}
    \le
    C e^{CA}
    \|u\|_{L^2(B_\rho(x_{j}'))}^{\alpha}
    \|u\|_{L^2(B_R)}^{1-\alpha},
\]
where \(C>0\) and \(\alpha\in(0,1)\) depend only on \(d,t\). Here we used that the
scale-normalized potential term on \(B_{4\rho}(x_j')\) is bounded by the corresponding
quantity on \(B_R\).

Iterating this estimate along the chain gives
\[
    \|u\|_{L^2(B_\rho(y))}
    \le
    C e^{CA}
    \|u\|_{L^2(B_\rho(x))}^{\gamma}
    \|u\|_{L^2(B_R)}^{1-\gamma},
\]
with
\(
    \gamma=\alpha^m.
\)
Since \(m\le C R/\rho\), the constants \(C\) and \(\gamma\) depend only on
\(d,t\) and \(\rho/R\). This proves \eqref{change:center}.

It remains to prove \eqref{lower:center}. We cover \(B_{R/2}\) by at most
\(C(R/\rho)^d\) balls of radius \(\rho\), with centers in \(B_{R/2}\). Hence there
is a center \(y\in B_{R/2}\) such that
\[
    \|u\|_{L^2(B_\rho(y))}
    \ge
    c
    \|u\|_{L^2(B_{R/2})},
\]
where \(c>0\) depends only on \(d\) and \(\rho/R\). Applying \eqref{change:center} we get
\[
    \|u\|_{L^2(B_\rho(x))}
    \ge
    C' e^{-CA/\gamma}
    \|u\|_{L^2(B_\rho(y))}^{1/\gamma}
    \|u\|_{L^2(B_R)}^{-(1-\gamma)/\gamma}.
\]
Using the lower bound for \(\|u\|_{L^2(B_\rho(y))}\) and the definition of \(N_R\), we
obtain
\[
    \|u\|_{L^2(B_\rho(x))}
    \ge
    e^{-C(1+A+N_R)}
    \|u\|_{L^2(B_{R/2})}.
\]
This proves the lemma.
\end{proof}

\subsection{Proof of Theorem \ref{th-1}}
Next we proceed to show the doubling inequality for the solutions of the Schr\"odinger equation (\ref{ePDE}). The argument is somewhat parallel to the proof of three-ball inequality. 

\begin{proof}[Proof of Theorem \ref{th-1}]
We give the proof for $R=8$; the general statement follows by
rescaling. Set
\[
A=A_V(B_8),\qquad
N=\log\frac{\|u\|_{L^2(B_8)}}{\|u\|_{L^2(B_4)}}.
\]

 Fix \(R_0=1/4\) and $x\in B_2$. The balls and annuli below are centered with $x$ and  \(A_{a,b}=B_b\setminus B_a\). We first consider \(0<r<R_0/12\). Choose
\(\psi\in C^\infty_0(B_{2R_0}\setminus\overline{B_{r/2}})\) such that
\(
\psi=1\) on \(A_{3r/4,\, R_0},
\)
and
\[
|\nabla\psi|\le Cr^{-1},\qquad
|\Delta\psi|\le Cr^{-2}
\quad\text{on }A_{r/2,\,3r/4},
\]
while
\[
|\nabla\psi|\le CR_0^{-1},\qquad
|\Delta\psi|\le CR_0^{-2}
\quad\text{on }A_{R_0,\,2R_0}.
\]
Thus
\[
\operatorname{supp}\nabla\psi\cup
\operatorname{supp}\Delta\psi
\subset A_{r/2, \, 3r/4}\cup A_{R_0,\, 2R_0},
\]
whereas
\[
A_{r,\, 2r}\cup A_{R_0/2, \, 2R_0/3}\subset\{\psi=1\}.
\]

Since \(u\) is a solution of \eqref{ePDE}, it holds that
\[
(\Delta+V)(\psi u)
=2\nabla\psi\cdot\nabla u+(\Delta\psi)u.
\]
The Caccioppoli estimate (\ref{Cacc-1}), applied on slightly larger annuli, gives
\begin{align}
r\|\nabla u\|_{L^2(A_{r/2, \, 3r/4})}
+\|u\|_{L^2(A_{r/2,\,3r/4})}
\le Ce^{CA}\|u\|_{L^2(B_r)}
\label{caccc-1}
\end{align}
and
\begin{align}
R_0\|\nabla u\|_{L^2(A_{R_0,\,2R_0})}
+\|u\|_{L^2(A_{R_0,\,2R_0})}
\le Ce^{CA}\|u\|_{L^2(B_{3R_0})}.
\label{caccc-2}
\end{align}

As in the proof of the three-ball inequality, we use the
potential-absorbed Carleman estimate which gives the exponent
\(\mu(d,t)\), or \(\mu_\varepsilon(2,t)\) when \(d=2\). Thus, in the
higher-integrability range, we use either the power-weight estimate
with exponent \(d\)  or the logarithmic-weight
estimate, according to which gives the smaller exponent. In what
follows, \(\mu\) denotes the corresponding exponent.

Applying Proposition~\ref{CarlpqV} with $v=\psi u$, by   H\"older's inequality and the preceding Caccioppoli estimates (\ref{caccc-1}) and (\ref{caccc-2}), for
every non-resonant \(\tau\ge CA\), we have
\begin{multline}
R_0^{-d/2}\left(\frac{2R_0}{3}\right)^{-\tau}
\|u\|_{L^2(A_{R_0/2,2R_0/3})}
+r^{-d/2}(2r)^{-\tau}\|u\|_{L^2(A_{r,2r})}
\\
\le Ce^{CA}\left[
r^{-d/2}\left(\frac r2\right)^{-\tau}
\|u\|_{L^2(B_r)}
+R_0^{-d/2}R_0^{-\tau}
\|u\|_{L^2(B_{3R_0})}
\right].
\label{eq:power-weight-doubling}
\end{multline}

Similarly, we apply  Proposition~\ref{proposi-2}  to $v= \psi u$ in
\(B_{3R_0}\). Write
\(\varphi=\varphi_{3R_0}\). Since \(\psi u\) vanishes in \(B_{r/2}\),
the parameter \(\rho\) in Proposition~\ref{proposi-2} is
\(r/2\). On \(A_{r,2r}\),
\[
\left(\frac r2\right)^{1/2}|z-x|^{-1/2}\ge\frac12,
\]
so the last term on the right-hand side of \eqref{Carl:ln} controls
the \(L^2\)-norm on \(A_{r,2r}\) without a factor involving
\(|\ell_{3R_0}|^{-1}\). Note that the factors
involving \(|\ell_{3R_0}|\) are bounded above and below by positive
constants in annulus $A_{R_0/2,\,2R_0/3}$
. We therefore obtain, for every \(\tau\ge CA\),
\begin{multline}
R_0^{-d/2}e^{\tau\varphi(2R_0/3)}
\|u\|_{L^2(A_{R_0/2,\,2R_0/3})}
+r^{-d/2}e^{\tau\varphi(2r)}
\|u\|_{L^2(A_{r,\,2r})}
\\
\le Ce^{CA}\left[
r^{-d/2}e^{\tau\varphi(r/2)}
\|u\|_{L^2(B_r)}
+R_0^{-d/2}e^{\tau\varphi(R_0)}
\|u\|_{L^2(B_{3R_0})}
\right].
\label{eq:log-weight-doubling}
\end{multline}

We next estimate the quotient of the two fixed-scale norms. Choose
\(\rho_0=R_0/24\) and a point \(y\) such that
\[
B_{\rho_0}(y)\subset A_{R_0/2,2R_0/3}.
\]
Applying Lemma~\ref{lem:change-center} in \(B_8\) gives that
\[
\|u\|_{L^2(A_{R_0/2,2R_0/3})}
\ge
\|u\|_{L^2(B_{\rho_0}(y))}
\ge
e^{-C(A+N)}\|u\|_{L^2(B_4)}.
\]
Since \(B_{3R_0}(x)\subset B_4\), it follows that
\[
\log
\frac{\|u\|_{L^2(B_{3R_0})}}
{\|u\|_{L^2(A_{R_0/2,2R_0/3})}}
\le C(A+N).
\]

For a power weight case, we choose a non-resonant \(\tau\) such that
\(
0\le \tau-C(A+N)\le 1,
\)
where \(C\) is sufficiently large. Since
\[
\frac{R_0^{-\tau}}
{(2R_0/3)^{-\tau}}
=\left(\frac23\right)^\tau,
\]
the second term on the right-hand side of
\eqref{eq:power-weight-doubling} is absorbed by the first term on its
left-hand side. Thus we obtain
\begin{align}
\|u\|_{L^2(A_{r,2r})}
\le Ce^{CA}4^\tau\|u\|_{L^2(B_r)}
\le e^{C(A+N)}\|u\|_{L^2(B_r)}.
\label{ccc-1}
\end{align}

For the logarithmic-weight $\varphi$,  we have
\(
c_0=\varphi(2R_0/3)-\varphi(R_0)>0.
\)
The number \(c_0\) is independent of \(R_0\), since the family of
weights \(\varphi_R\) is compatible with dilations. Choosing
\(\tau=C(A+N)\) with \(C\) sufficiently large, we absorb the outer
term in \eqref{eq:log-weight-doubling}. It follows that
\[
\|u\|_{L^2(A_{r,2r})}
\le
Ce^{CA}
e^{\tau(\varphi(r/2)-\varphi(2r))}
\|u\|_{L^2(B_r)}.
\]
Moreover,
\[
\varphi(r/2)-\varphi(2r)
=\int_{r/2}^{2r}-\varphi'(s)\,ds
\le C\int_{r/2}^{2r}\frac{ds}{s}
\le C,
\]
uniformly for \(0<r<R_0/12\). Hence
\begin{align}
\|u\|_{L^2(A_{r,2r})}
\le e^{C(A+N)}\|u\|_{L^2(B_r)}.
\label{ccc-2}
\end{align}
Note that the constant $C$ in (\ref{ccc-1}) and \eqref{ccc-2} are independent of $r$.

Since
\[
\|u\|_{L^2(B_{2r})}
\le
\|u\|_{L^2(B_r)}
+\|u\|_{L^2(A_{r,2r})},
\]
we conclude from (\ref{ccc-1}) and (\ref{ccc-2})  that
\[
\|u\|_{L^2(B_{2r})}
\le e^{C(A+N)}\|u\|_{L^2(B_r)}
\]
for \(0<r<R_0/12\).

It remains to consider \(r\ge R_0/12\). Since
\(B_{\rho_0}(x)\subset B_r(x)\), Lemma~\ref{lem:change-center}
gives
\[
\|u\|_{L^2(B_r(x))}
\ge
\|u\|_{L^2(B_{\rho_0}(x))}
\ge
e^{-C(A+N)}\|u\|_{L^2(B_4)}.
\]
It holds that \(B_{2r}(x)\subset B_8\) for \(x\in B_2\) and \(r\le2\). 
Therefore
\[
\|u\|_{L^2(B_{2r}(x))}
\le
\|u\|_{L^2(B_8)}
=e^N\|u\|_{L^2(B_4)}.
\]
Combining the last two inequalities yields that
\[
\|u\|_{L^2(B_{2r}(x))}
\le
e^{C(A+N)}\|u\|_{L^2(B_r(x))}.
\]
This completes the proof.
\end{proof}

\section{Propagation of smallness from sets of positive measure}\label{s:Remez}
\label{Remez}
\subsection{Doubling indices for cubes}
In this section we prove a propagation of smallness estimate from sets of positive measure for solutions of \eqref{ePDE}.
The question of quantitative propagation of smallness for solutions of elliptic equations with lower-order terms was studied in \cite{MV12}, where a weaker estimate was obtained using Carleman inequalities of Koch and Tataru \cite{KT01}. In this work we consider the Schr\"odinger equations with singular potentials and give explicit estimates on the constants in our inequalities.

We adapt the strategy in \cite{LM18}, \cite{LM19}.  However, we work with complex-valued solutions and complex-valued, possibly singular potentials and use the results of \cite{Kato78} and  \cite{BK79} for the base of our induction argument.

Throughout this section \(d\ge2\), \(t>d/2\), \(u\) is a solution
of \eqref{ePDE} in a cube \(Q_0\), and \(V\in L^t(Q_0)\). If \(d\ge3\),
we write \(\mu=\mu(d,t)\), as defined in \eqref{mu3}. If \(d=2\),
we fix \(\varepsilon>0\) and write
\(\mu=\mu_\varepsilon(2,t)\), as defined in \eqref{mu2}. The solution
\(u\) and the potential \(V\) may be complex-valued.

If \(q\) is a cube, then \(2q\) denotes the cube with the same center
and twice the side length.  We denote the side length of \(q\) by
\(s(q)\). We will also use notation $Q_r(z)$ for a cube with side lengths $r$ and center at $z$.
For a cube \(Q\subset Q_0\), define
\[
\mathcal{A}_V(Q)=
\left(
1+s(Q)^{2-d/t}\|V\|_{L^t(Q)}
\right)^\mu.
\]
Also when \(2Q\subset Q_0\), define
\[
N_u(Q)=\log\frac{\sup_{2Q}|u|}{\sup_Q|u|}.
\]

We first compare the \(L^2\)- and \(L^\infty\)-norms. The elementary
inequality
\[
r^{-d/2}\|u\|_{L^2(B_r(z))}
\le C\|u\|_{L^\infty(B_r(z))}
\]
and Lemma~\ref{lem:L2-Linfty-local} imply that, whenever
\(B_{2r}(z)\subset Q_0\),
\begin{equation}
c r^{-d/2}\|u\|_{L^2(B_r(z))}
\le \|u\|_{L^\infty(B_r(z))}
\le
C e^{C A_V(B_{2r}(z))}
r^{-d/2}\|u\|_{L^2(B_{2r}(z))}.
\label{eq:L2-Linfty-comparison}
\end{equation}
Indeed, we use the fact that the polynomial factor in Lemma~\ref{lem:L2-Linfty-local}
is bounded by \(Ce^{C A_V(B_{2r}(z))}\).

We next compare the fixed-scale \(L^2\)-growth with the doubling
index of a cube. Let \(Q=Q_R(x)\). Assume that
\(\Lambda_dQ\subset Q_0\), where \(\Lambda_d>1\) is a sufficiently large dimensional constant. We denote
\[
A=\mathcal{A}_V(\Lambda_dQ),\qquad N=N_u(Q).
\]
There exists a dimensional constant \(\rho_d>0\) such that, for
\(\rho=\rho_dR\) and every \(z\in Q\),
\begin{equation}
\log
\frac{\|u\|_{L^2(B_{8\rho}(z))}}
{\|u\|_{L^2(B_{4\rho}(z))}}
\le C(A+N).
\label{eq:fixed-scale-L2-growth}
\end{equation}

To see this, choose \(z_0\in Q\) such that
\[
|u(z_0)|\ge\frac12\sup_Q|u|.
\]
Taking \(\rho_d\) to be small, all the balls used below are
contained in \(2Q\). By \eqref{eq:L2-Linfty-comparison},
\[
\|u\|_{L^2(B_{2\rho}(z_0))}
\ge c e^{-CA}R^{d/2}\sup_Q|u|.
\]
We connect \(z_0\) to \(z\) by a chain of points
\(z_0,z_1,\ldots,z_m=z\) contained in \(Q\), where
\(|z_{j+1}-z_j|\le\rho\) and \(m\le C_d\). Applying Lemma~\ref{lem:three-ball} at \(z_j\), with radii
\(\rho<2\rho<8\rho\), and using
\(B_\rho(z_j)\subset B_{2\rho}(z_{j+1})\), we obtain
\[
\|u\|_{L^2(B_{2\rho}(z_j))}
\le
Ce^{CA}
\|u\|_{L^2(B_{2\rho}(z_{j+1}))}^{\alpha}
\left(R^{d/2}\sup_{2Q}|u|\right)^{1-\alpha},
\]
where \(\alpha\in(0,1)\) depends only on the dimension. Iterating
along the chain and using
\(\sup_{2Q}|u|=e^N\sup_Q|u|\), we get
\[
\|u\|_{L^2(B_{4\rho}(z))}
\ge
e^{-C(A+N)}R^{d/2}\sup_Q|u|.
\]
On the other hand,
\[
\|u\|_{L^2(B_{8\rho}(z))}
\le CR^{d/2}\sup_{2Q}|u|.
\]
This proves \eqref{eq:fixed-scale-L2-growth}.

Applying Theorem~\ref{th-1} in \(B_{8\rho}(z)\), and using
\eqref{eq:fixed-scale-L2-growth}, we conclude that
\begin{equation}
\log
\frac{\|u\|_{L^2(B_{2s}(z))}}
{\|u\|_{L^2(B_s(z))}}
\le C(A+N),
\qquad 0<s\le2\rho.
\label{eq:small-scale-L2-doubling}
\end{equation}

We now pass from balls back to cubes. Since
\(
B_{r/2}(z)\subset Q_r(z)
\) and
\(2Q_r(z)\subset B_{\sqrt d\,r}(z),
\)
the comparison \eqref{eq:L2-Linfty-comparison} gives
\begin{equation}
N_u(Q_r(z))
\le
CA+
\log
\frac{\|u\|_{L^2(B_{2\sqrt d\,r}(z))}}
{\|u\|_{L^2(B_{r/2}(z))}}.
\label{eq:cube-ball-doubling-comparison}
\end{equation}
If \(r\le c_dR\), where \(c_d>0\) is  small, the
quotient on the right-hand side is a product of a fixed number,
depending only on \(d\), of doubling quotients covered by
\eqref{eq:small-scale-L2-doubling}. Consequently,
\begin{equation}
N_u(Q_r(z))
\le C\bigl(N_u(Q)+\mathcal{A}_V(\Lambda_d Q)\bigr)
\label{eq:deep-cube-comparison}
\end{equation}
for every \(z\in Q\) and every \(0<r\le c_dR\).

The results of the previous sections imply the following almost monotonicity lemmas for the doubling indices. 

\begin{lemma}
\label{lem:deep-inside-comparison}
There exists \(\eta=\eta(d)\in(0,1)\), constants \(C_1>0\) and \(A_1>0\), depending only on
\(d,t\), such that if \(q(y)\) and \(Q(x)\) are cubes satisfying
\(
    q(y)\subset \eta Q(x),\)
    \(
    \Lambda_d Q(x)\subset Q_0,
\)
then
\[
    N_u(q(y))\le C_1N_u(Q(x))+A_1\mathcal{A}_V(\Lambda_dQ(x)).
\]
\end{lemma}

We define the maximal doubling index in a cube $Q$ with $2Q\subset Q_0$, by  
\begin{align}
\mathcal{N}_u(Q)=\sup_{q\subset Q} N_u(q).
    \label{eq:mondouble}
\end{align}
From its definition, it is clear that   $\mathcal{N}_u(q)\le \mathcal{N}_u(Q)$ for $q\subset Q$. 

\begin{corollary}
\label{cor:maximal-index-control}
Let \(\eta\) and $C_1$ be the constants in Lemma~\ref{lem:deep-inside-comparison}.
Define $\widehat{\Lambda}_d=\eta^{-1}\Lambda_d$  and suppose that
\(
    \widehat{\Lambda}_d Q\subset Q_{0},
\)
then there exists $A_1$ depending only on $t,d$ such that
\[
    \mathcal N_u(Q)\le C_1N_u(\eta^{-1} Q)+A_1\mathcal{A}_V(\widehat{\Lambda}_dQ).
\]
In particular, \(\mathcal N_u(Q)<\infty\).
\end{corollary}

\begin{proof}
Let \(q\subset Q\). Our aim is to prove a bound for
\(N_u(q)\) independent of \(q\).
By assumption, 
\(
    \Lambda_d \eta^{-1}Q\subset Q_{0}\).
We apply Lemma~\ref{lem:deep-inside-comparison} to the cubes $q$ and $\eta^{-1}Q$ to obtain that
\[
    N_u(q)\le C_1N_u(\eta^{-1} Q)+A_1(1+\eta^{d/t-2}s(Q)^{2-d/t}\|V\|_{L^t(\widehat{\Lambda}_dQ)})^\mu.
\]
Taking the supremum over all \(q\subset Q\) 
 proves the result.
\end{proof}

\subsection{Reformulation and induction bases}
	To prove the propagation of smallness estimate, we will establish a Remez type inequality as in \cite{LM18}. First we set $\widetilde{\Lambda}_d=2\widehat{\Lambda}_d$ and define a modified doubling index which also takes the norm of the potential into account as
\[\mathcal{M}_u(Q) =\mathcal{N}_u(2Q)+\mathcal{A}_V(\widetilde\Lambda_dQ).\]
Since $t>d/2$, it is easy to see that for $q\subset Q$,
\(\mathcal{M}_u(q)\le \mathcal{M}_u(Q).\)
Corollary \ref{cor:maximal-index-control} implies that 
\begin{equation}
    \label{max-doubling-doubling}
\mathcal{M}_u(Q)\le C\left(N_u(2\eta^{-1}Q)+\mathcal A_V(\widetilde{\Lambda}_dQ)\right).\end{equation}

	\begin{proposition}
	 \label{prop:R}  Let $u$ be a solution of (\ref{ePDE}) in $Q_0$. Suppose that $\widetilde\Lambda_dQ\subset Q_0$ and denote 
		\(E_a(u)=\{x\in Q: |u(x)|<e^{-a}\sup_Q|u|\}.\)
		Then 
		\begin{equation}\label{eq:RPDE}
			|E_a(u)|\le Ce^{-\frac{\gamma a} {{\mathcal{M}}_u(Q)}}|Q|
		\end{equation}
		for some positive constants $C$ and $\gamma$ that depend on $d$ and $t$  only. 
        \label{pro-4}
	\end{proposition} 
	
In order to prove the above proposition, we apply the double induction argument which will be explained in this and the next subsections. First of all, we need to verify the base cases of induction  in $a$ and ${\mathcal{M}}=\mathcal{M}_u(Q)$. We consider 
	two cases  $a\le c_0 {\mathcal{M}}$ and ${\mathcal{M}}\le N_0$ for some $N_0$ large enough.

    If $\frac{a}{ {\mathcal{M}}} <c_0$, by choosing $C=C(\gamma)$ large enough, we get
\(Ce^{-\gamma \frac{a}{{\mathcal{M}}}}\ge Ce^{-\gamma c_0}\ge 1.
    \)
    Thus,  the inequality (\ref{eq:RPDE}) holds trivially.

	Now we want to show that (\ref{eq:RPDE}) holds for some $\gamma$ and $C$ if we assume that ${\mathcal{M}_u(Q)}\leq N_0$. We apply  the oscillation  Lemma \ref{lem:complex-oscillation} in Appendix. 
    
	\begin{lemma}
\label{lem:bounded-index-base-case}
 Assume that $u$ satisfies (\ref{ePDE}) in $Q_0$, $\widetilde\Lambda_dQ\subset Q_0$ and $\mathcal{M}_u(Q) \leq N_0$. Let 
		$E_a=\{x\in Q: |u(x)|<e^{-a}\sup_Q|u|\}.$ 
		Then there exist constants \(C>0\) and \(\gamma>0\), depending only on
\(N_0,d,t\), such that
\[
    |E_a(u)|\le C e^{-\gamma a}|Q|.
\]
\end{lemma}

\begin{proof}
We first prove the following scale-invariant claim, following \cite{LM18}.
{\it There exist an integer \(K\ge2\) and constants \(b_0,b_1\in(0,1)\), depending only on
\(N_0,d,t\), with the following property. Let \(Q_1\subset Q\) be a cube such that
\(
    \mathcal M_u(Q_1)\le N_0,\)
    \(
    m=\sup_{Q_1} |u|>0.
\)
Partition \(Q_1\) into \(K^d\) equal subcubes. Then
there is one subcube $q_0\subset Q_1$
    such that 
    $\inf_{q_0}|u|\ge b_0m,$
and every subcube \(q\subset Q_1\) in the partition satisfies
\(
    \sup_q |u|\ge b_1m.
\)
}

We prove the claim. The definition of the  doubling index gives
\[
    N_u\left(\frac12 Q_1\right)
    =
    \log\frac{\sup_{Q_1}|u|}{\sup_{\frac12 Q_1}|u|}
    \le
    \mathcal N_u(Q_1)
    \le N_0.
\]
Hence
\[
    \sup_{\frac12 Q_1}|u|\ge e^{-N_0}m.
\]
Choose \(x_0\in \frac12 Q_1\) such that
\[
    |u(x_0)|\ge e^{-N_0}m.
\]
Let \(q_0\) be the subcube of the partition containing \(x_0\).
Since \(x_0\in \frac12 Q_1\), a ball with radius $r=c_ds(Q_1)$ and
center \(x_0\) is contained in \(Q_1\). Applying Lemma \ref{lem:complex-oscillation}
 to this ball, we obtain
\[
    \operatorname{osc}^{\mathbb C}_{q_0}u
    \le
    C(d,t,N_0)K^{-\alpha}m.
\]
Choosing \(K\) sufficiently
large, depending only on \(N_0,d,t\), we may assume
\[
    C(d,t,N_0)K^{-\alpha}\le \frac12 e^{-N_0}.
\]
Therefore, for every \(x\in q_0\),
\[
    |u(x)|
    \ge
    |u(x_0)|-|u(x)-u(x_0)|
    \ge
    \frac12 e^{-N_0}m.
\]
Thus the first assertion of the claim holds with
\(b_0=\frac12 e^{-N_0}.\)

It remains to prove the lower bound on \(\sup_q |u|\) for every subcube \(q\) of the
partition. Let
\( h=K^{-1}{s(Q_1)}
\)
be the side length of the subcubes in the partition. Choose a point \(y\in Q_1\) such that
\(
    |u(y)|\ge \frac{m}{2},
\)
and let \(q_\ast\) be a subcube of the partition which contains \(y\). 

Fix now an arbitrary subcube \(q\) of the partition. There exists a chain of subcubes of the partition
\(
    q=q^{(0)},q^{(1)},\ldots,q^{(J)}=q_\ast
\)
such that consecutive cubes \(q^{(j)}\) and \(q^{(j+1)}\) share a face, and
\(    J\le d(K-1).
\)

We claim that for every \(j=0,\ldots,J-1\),
\[
    \sup_{q^{(j+1)}}|u|
    \le
    e^{2N_0}\sup_{q^{(j)}}|u|.
\]
Indeed, since \(q^{(j)}\) and \(q^{(j+1)}\) are adjacent equal cubes, we have
\(
    q^{(j+1)}\subset 4q^{(j)}.
\)
Moreover, since \(4q^{(j)}\subset 2Q_1\), and
\(\mathcal N_u(2Q_1)\le N_0\), we have
\(
    N_u(q^{(j)})\le N_0,\ N_u(2q^{(j)})\le N_0.
\)
Therefore
\[
    \sup_{q^{(j+1)}}|u|
    \le
    \sup_{4q^{(j)}}|u|
    \le
    e^{2N_0}\sup_{q^{(j)}}|u|
.
\]
Iterating along the chain gives
\[
    \frac{m}2 
    \le
    \sup_{q_\ast}|u|
    \le
    e^{2JN_0}\sup_q |u|
    \le
    e^{2d(K-1)N_0}\sup_q |u|.
\]
Hence
\[
    \sup_q |u|\ge b_1 m,
    \qquad
    b_1=\frac12 e^{-2d(K-1)N_0}.
\]
This proves the second assertion of the claim.
We now apply the claim iteratively  to complete the proof, a similar argument can be found in as in \cite{LM18}.


\end{proof}

\subsection{Induction step}
 The induction will be carried out with
respect to the modified doubling index \(\mathcal M_u\), rather than the maximal doubling
index alone. Recall that \(\mathcal M_u\) is defined by
\(
    \mathcal M_u(Q)
    =
    \mathcal{N}_u(2Q)+\mathcal{A}_V(\widetilde\Lambda_dQ).
\)

We prove  the following good-subcube lemma.

\begin{lemma}
\label{lem:good-subcube-modified}
There exist a dyadic integer \(K=2^\kappa\ge2\), constants \(N_0>1\) and
\(a_0>0\), depending only on \(d,t\), such that the following holds.
Suppose that \(\widetilde\Lambda_dQ\subset Q_0\)  and 
\(
    \mathcal M_u(Q)=\widetilde{\mathcal{M}}\ge N_0.
\)
Partition \(Q\) into \(K^d\) equal subcubes. Then at least one subcube \(q_0\)
satisfies
\(
    \mathcal M_u(q_0)\le \frac{\widetilde{\mathcal{M}}}{2}.
\)
Moreover, every subcube \(q\) of the partition satisfies
\(
    \sup_q |u|
    \ge
    e^{-a_0\widetilde{\mathcal{M}}}\sup_Q |u|.
\)
\end{lemma}

\begin{proof}
We first prove the existence of a subcube with smaller modified
doubling index. Let \(\eta\) be the constant in
Corollary~\ref{cor:maximal-index-control}, and write
\(
p=2\eta^{-1}q
\) for each subcube $q$ of the partition.
Taking \(K\) sufficiently large, depending only on the dimension,
we may assume that all the cubes \(2p_i\) are contained in \(2Q\).

Let $N_{min}=N(p_0)$ be the smallest doubling index for cubes $p$. 
Let $x_0\in Q/2$ be such that $|u(x_0)|=\sup_{Q/2}|u|$. Then $x_0\in q\subset p$ for some $q$ and then there exists $x_1$ such that $ |x_0-x_1|\le CK^{-1}s(Q)$ such that 
$|u(x_1)|\ge e^{N_{min}}|u(x_0)|.$
Then, repeating this step
\(cK\) times, we obtain
\[
\sup_{Q}|u|
\ge
e^{cKN_{min}}\sup_{Q/2}|u|.
\]
 The argument is similar to one in \cite[Lemma~4.1.1]{LM19}.
Consequently,
\[
N(p_0)
\le
\frac{C}{K}
\log\frac{\sup_{Q}|u|}{\sup_{Q/2}|u|}
\le
\frac{C}{K}\mathcal N_u(Q).
\]
We now apply Corollary~\ref{cor:maximal-index-control} to \(2q_0\). It implies 
\[
\mathcal N_u(2q_0)
\le
C_1N_u(p_0)+C_2\mathcal{A}_V(\widetilde\Lambda_dq_0)
\le
\frac{C_3}{K}\mathcal N_u(Q)
+C_2\mathcal{A}_V(\widetilde\Lambda_d q_0)
\]

We next compare the potential terms. Since \(q_0\subset Q\), we have
\[
\widetilde\Lambda_dq_0\subset\widetilde\Lambda_dQ,
\qquad
s(\widetilde\Lambda_dq_0)
=
K^{-1}s(\widetilde\Lambda_dQ).
\]
Therefore
\[
s(\widetilde\Lambda_dq_0)^{2-d/t}
\|V\|_{L^t(\widetilde\Lambda_dq_0)}
\le
K^{-(2-d/t)}
s(\widetilde\Lambda_dQ)^{2-d/t}
\|V\|_{L^t(\widetilde\Lambda_dQ)}.
\]
Since \(t>d/2\), setting
\[
\sigma=\mu\left(2-\frac dt\right)>0,
\]
we obtain
\[
\mathcal A_V(\widetilde\Lambda_dq_0)
\le
C\left(
1+K^{-\sigma}
\mathcal A_V(\widetilde\Lambda_dQ)
\right).
\]

Therefore
\[
\mathcal M_u(q_0)=
\mathcal N_u(2q_0)+\mathcal{A}_V(\widetilde\Lambda_dq_0)\le
\frac{C_3}{K}\mathcal N_u(2Q)
+C_5K^{-\sigma}\mathcal{A}_V(\widetilde\Lambda_dQ)+C_6.
\]
We choose the dyadic integer \(K\) sufficiently large so that
\(
\frac{C_3}{K}\le\frac14,\) and 
\(
C_5K^{-\sigma}\le\frac 14.
\)
Then
\(
\mathcal M_u(q_0)
\le
\frac14\widetilde{\mathcal M}+C_6\le \frac12\widetilde{\mathcal M},
\)
when $\widetilde{M}\ge N_0$ and $N_0$ is large enough.

It remains to prove the lower bound for the supremum on each subcube. Let \(q\) be an
arbitrary subcube of the partition. Since \(K=2^\kappa\), consider the concentric dyadic
dilates
\(
    q^{(j)}=2^j q,\)
for \(j=0,1,\ldots,\kappa+1.
\)
Then \(q^{(0)}=q\), the cube \(q^{(\kappa)}=Kq\) has the same side length as \(Q\), and
\(q^{(\kappa+1)}=2Kq\) contains \(Q\). Moreover, for \(j=0,\ldots,\kappa\), the cube
\(q^{(j)}\) is contained in \(2Q\).
Since \(Q\subset q^{(\kappa+1)}\), by Corollary~\ref{cor:maximal-index-control},  we obtain
\[
    \sup_Q|u|
    \le
    e^{(\kappa+1)\widetilde{\mathcal{M}}}\sup_q|u|.
\]
This proves the lemma.
\end{proof}

We now complete the proof of Proposition \ref{prop:R}.
\begin{proof}[Proof of  Proposition \ref{prop:R}]
For \(\widetilde{\mathcal{M}}\ge1\) and \(a>0\), define
\[
    F(\widetilde{\mathcal{M}},a)
    =
    \sup
    \frac{
    \left|
    \left\{
        x\in Q:\ |u(x)|<e^{-a}\sup_Q|u|
    \right\}
    \right|
    }{|Q|},
\]
where the supremum is taken over all cubes \(Q\), all potentials \(V\), and all
solutions \(u\) of 
\[
    \Delta u+V(x)u=0
\]
in $2\Lambda_dQ$ such that
\(
    \mathcal M_u(Q)\le \widetilde{\mathcal{M}}.
\)
This definition is invariant under rescaling of \(Q\).

We claim that
\[
    F(\widetilde{\mathcal{M}},a)\le C e^{-\gamma a/\widetilde{\mathcal{M}}}
\]
for all \(\widetilde{\mathcal{M}}\ge1\) and all \(a>0\), where \(C,\gamma>0\) depend only on \(d,t\). Following the argument in \cite{LM18},
it is easy to see that  the good-subcube lemma (i.e. Lemma\ref{lem:good-subcube-modified}) above implies the following recursive inequality.
\if false
We first derive the recursive inequality. Assume that
\(
    \widetilde{\mathcal{M}}\ge N_0,
    \
    a>a_0\widetilde{\mathcal{M}}.
\)
Let \(Q\) and \(u\) be as above
\(
    \mathcal M_u(Q)\le\widetilde{\mathcal{M}},
\)
and define
\(
    m_Q=\sup_Q|u|.
\)
Partition \(Q\) into \(K^d\) equal subcubes. By
Lemma~\ref{lem:good-subcube-modified}, there is one subcube \(q_0\) such that
\[
    \mathcal M_u(q_0)\le\frac{\widetilde{\mathcal{M}}}{2},
\]
and every subcube \(q\) satisfies
\[
    \sup_q|u|\ge e^{-a_0\widetilde{\mathcal{M}}}m_Q.
\]
Therefore, for every subcube \(q\),
\[
    \{x\in q:\ |u(x)|<e^{-a}m_Q\}
    \subset
    \{x\in q:\ |u(x)|<e^{-a+a_0\widetilde{\mathcal{M}}}\sup_q|u|\}.
\]
The good subcube contributes at most
\[
    |q_0|F\left(\frac{\widetilde{\mathcal{M}}}{2},a-a_0\widetilde{\mathcal{M}}\right),
\]
and the remaining \(K^d-1\) subcubes contribute at most
\[
    \sum_{q\ne q_0}|q|F(\widetilde{\mathcal{M}},a-a_0\widetilde{\mathcal{M}}).
\]
Dividing by \(|Q|\), and using \(|q|=K^{-d}|Q|\), we obtain
\[
    F(\widetilde{\mathcal{M}},a)
    \le
    K^{-d}F\left(\frac{\widetilde{\mathcal{M}}}{2},a-a_0\widetilde{\mathcal{M}}\right)
    +
    (1-K^{-d})F(\widetilde{\mathcal{M}},a-a_0\widetilde{\mathcal{M}}).
\]
Thus
\fi
\[
\label{eq:F-recursion}
    F(\widetilde{\mathcal{M}},a)
    \le
    (1-s)F\left(\frac{\widetilde{\mathcal{M}}}{2},a-a_0\widetilde{\mathcal{M}}\right)
    +
    sF(\widetilde{\mathcal{M}},a-a_0\widetilde{\mathcal{M}}),
\]
where $s=1-K^{-d}<1.$

 After establishing the recursion inequality, we complete the proof of (\ref{eq:RPDE}) as in \cite{LM18}.
\end{proof}
\subsection{Generalized three ball inequality} We show that the Remez inequality obtained above implies the propagation of smallness from sets of positive measure. 
\begin{corollary}
\label{cor:measurable-three-cube}
Let \(Q\) be a cube and assume that
\[
    \Delta u+V(x)u=0
    \qquad\text{in }\widetilde\Lambda_d Q,
\]
where \(V\in L^t(\widetilde\Lambda_d Q)\), \(t>d/2\). 
Then for every measurable set \(E\subset Q\) with \(|E|>0\), one has
\[
    \sup_Q |u|
    \le
    \left(
        C\frac{|Q|}{|E|}
    \right)^{
        C\alpha\left(1+\mathcal{A}_V( \widetilde\Lambda_d Q)\right)
    }
    \left(\sup_E |u|\right)^\alpha
    \left(\sup_{\widetilde\Lambda_dQ}|u|\right)^{1-\alpha},
\]
where
\[
    \alpha
    =
    \left(
        1+
        C\log\left(C\frac{|Q|}{|E|}\right)
    \right)^{-1},
\]
and the constant \(C\) depends only on \(d,t\), and on the fixed $\varepsilon>0$ for the case $d=2$.
\end{corollary}

\begin{remark}
    This is a version of Theorem \ref{th-2}. The full statement of Theorem \ref{th-2} can be deduced from this corollary by iteration.
\end{remark}

\begin{proof}
We normalize the solution by $\sup_{\widetilde\Lambda_d Q}|u|=1$ and denote
\(  m_E=\sup_E |u|,\) and 
    \(m_Q=\sup_Q |u|\).
The inequality \eqref{eq:RPDE} readily implies
\[
    m_Q
    \le
    m_E
    \left(
        C\frac{|Q|}{|E|}
    \right)^{C\mathcal M_u(Q)}.
\]
Then \eqref{max-doubling-doubling} also gives
\[
    \mathcal M_u(Q)
    \le
    C\left(-\log m_Q+
        \mathcal{A}_V(\widetilde\Lambda_d Q)
    \right).
\]
Therefore, setting
\(
    D=C\frac{|Q|}{|E|},
\)
we obtain
\[
    m_Q
    \le
    m_E
    D^{C(1+\mathcal{A}_V(\widetilde\Lambda_d  Q))}
    m_Q^{-C\log D}.
\]
It follows that
\[
    m_Q^{1+C\log D}
    \le
    m_E
    D^{C(1+\mathcal{A}_V(\widetilde\Lambda_d Q))}.
\]
Raising both sides to the power
\(
    \alpha
    =
    \left(1+C\log D\right)^{-1}
\)
gives
the desired estimate.
\end{proof}

\subsection{Application to Dirichlet eigenfunctions in planar Lipschitz domains}

We now describe an application of the two-dimensional version of the preceding results.
Let \(\Omega\subset\mathbb C\) be a bounded simply connected Lipschitz domain.  We denote by
\(  F:\mathbb D\to\Omega
\)
a conformal map.

We recall that the Dirichlet Laplacian on
\(\Omega\) has
a discrete spectrum and refer the reader to Davies~\cite{Davies89} for the spectral-theoretic
background. A Dirichlet Laplace eigenfunction \(u_\lambda\) satisfies
\[
    u_\lambda\in H^1_0(\Omega),
    \qquad
    \int_\Omega \nabla u_\lambda\cdot\nabla \varphi\,dx
    =
    \lambda\int_\Omega u_\lambda\varphi\,dx
\]
for every \(\varphi\in H^1_0(\Omega)\).

We also use the following fact about the conformal map. Since \(\Omega\) is a bounded  simply connected planar Lipschitz domain, it is a
quasidisk, see  Pommerenke~\cite{Pommerenke92} and Gehring~\cite{Gehring73}. Hence the  map \(F:\mathbb D\to\Omega\) admits a quasiconformal
extension to the plane. By the  theorem on integrability for quasiconformal maps, \cite[Theorem 1]{Gehring73},
there exists \(p_\Omega>2\), depending only on  \(\Omega\),
such that
\(
    F'\in L^{p_\Omega}(\mathbb D).
\)
 We fix
\(t_\Omega\in(1,2]\) such that
\[
    V(z)=|F'|^2\in L^{t_\Omega}(\mathbb D).
\]

Let
\(
    v_\lambda(z)=u_\lambda(F(z)).
\)
By conformal invariance of the Dirichlet integral in dimension two,
\(
    v_\lambda\in H^1_0(\mathbb D)
\)
and  satisfies
\[
    \Delta v_\lambda
    +
    \lambda V(z) v_\lambda=0
\]
in \(\mathbb D\) in the weak sense.  
Since \(v_\lambda\in H^1_0(\mathbb D)\), we may reflect it oddly across
\(\partial\mathbb D\). More explicitly,  define
\[
    \widetilde v_\lambda(z)
    =
    \begin{cases}
        v_\lambda(z), & |z|<1,\\
        -v_\lambda(1/\overline z), & |z|>1.
    \end{cases}
\]
Then \(\widetilde v_\lambda\) is a weak solution of
\(
    \Delta \widetilde v_\lambda+\lambda \widetilde V\widetilde v_\lambda=0
\)
in \(\mathbb C\), where
\[
    \widetilde V(z)
    =
    \begin{cases}
        |F'(z)|^2, & |z|<1,\\
        |z|^{-4} |F'(1/\overline z)|^2, & |z|>1.
    \end{cases}
\]
We also have
\begin{align*}
    \int_{\mathbb C}|\widetilde V(z)|^{t_\Omega}dm_2(z)&=\int_{|z|<1}|F'(z)|^{2t_\Omega}dm_2(z)+\int_{|z|>1}|z|^{-4t_\Omega}|F'(1/\overline z)|^{2t_\Omega}dm_2(z)\\&=\int_{|z|<1}|F'(z)|^{2t_\Omega}(1+|z|^{4t_\Omega-4})dm_2(z)\le 2\int_{|z|<1}|F'(z)|^{2t_\Omega}.\end{align*}
Therefore  the reflected potential still belongs to \(L^{t_\Omega}\).

\begin{proposition}
\label{prop:BMO-Lipschitz-domain}
Let \(\Omega\subset\mathbb C\) be a bounded simply connected Lipschitz domain, and let $\varepsilon>0$. Then there exists $C_{\Omega,\varepsilon}$ such that for any
  Dirichlet Laplace eigenfunction \(u_\lambda\)  in \(\Omega\), we have 
\[
    \|\log |u_\lambda|\|_{\mathrm{BMO}(\Omega)}
    \le
    C_{\Omega,\varepsilon}
    \left(1+\lambda^{\mu_\varepsilon(2, t_\Omega)}\right).
\]
\end{proposition}

\begin{proof}
Denoted  
\(\widetilde v_\lambda\) be as above. Let $Q_0$ be a square containing $\mathbb D$ and let $\widetilde{D}$ be a large fixed disk centered at the origin such that for any square $\widetilde\Lambda_2 Q_0\subset\widetilde D$. We prove that
for every \(\varepsilon>0\),
\[
    \|\log |v_\lambda|\|_{\mathrm{BMO}(\mathbb D)}
    \le
    C_{\Omega,\varepsilon}
    \left(1+\lambda^{\mu_\varepsilon(2, t_\Omega)}\right),
\]
where
\[
    \mu_\varepsilon(2, t_\Omega)
    =
    \frac{t_\Omega}{2(t_\Omega-1)}+\varepsilon .
\]
 The estimate for \(u_\lambda\) follows from
the quasiconformal invariance of BMO, since \(F\) extends quasiconformally to the plane, see e.g. \cite{R74}.

Normalize
\(
    \sup_{\widetilde{D}}|\widetilde{v}_\lambda|=\sup_{\mathbb D}|v_\lambda|=1.
\)
The reflected function \(\widetilde v_\lambda\) solves the Schrödinger equation in $\widetilde{D}$, with potential bounded in
\(L^{t_\Omega}\) by \(C_\Omega\lambda\). Hence the constant $\mathcal A_V(\widetilde{\Lambda}_2Q)$ which 
appears in the doubling and Remez estimates is bounded by
\(
    C_{\Omega,\varepsilon}
    \left(1+\lambda^{\mu_\varepsilon(2, t_\Omega)}\right).
\)
We also see that $N_{\widetilde{v}_\lambda}(Q_1)=0$ for any $Q_1\supset \mathbb D$.

Then Lemma \ref{lem:deep-inside-comparison} implies \(N_{\widetilde{v}_\lambda}(q)\le A_1\mathcal A_{\lambda\widetilde{V}}(\widetilde\Lambda_2Q_0),\) for any $q\subset Q_0$. Then 
\[\widetilde{\mathcal M}=\mathcal{M}_{\widetilde{v}_\lambda}(Q_0)\le C_{\Omega,\varepsilon}
    \left(1+\lambda^{\mu_\varepsilon(2, t_\Omega)}\right).\]
    By Proposition \ref{prop:R}, for every \(a>0\), we obtain
\[
\left|
\left\{
z\in Q:
|\widetilde v_\lambda(z)|
<
e^{-a}\sup_Q|\widetilde v_\lambda|
\right\}
\right|
\le
Ce^{-ca/\widetilde{\mathcal M}}|Q|.
\]
Set
\(
f=\log|\widetilde v_\lambda|,\) \(M_Q=\log\sup_Q|\widetilde v_\lambda|.\)
The preceding estimate implies that \(f\in L^1(Q)\) and
\[
|\{z\in Q:M_Q-f(z)>a\}|
\le
Ce^{-ca/\widetilde{\mathcal{M}}}|Q|.
\]
Integrating this estimate with respect to \(a\), we obtain
\[
\fint_Q(M_Q-f)\,dm_2
\le
C\widetilde{\mathcal M}.
\]
Since \(f\le M_Q\) on \(Q\),
\[
\fint_Q|f-f_Q|\,dm_2
\le
2\fint_Q(M_Q-f)\,dm_2
\le
C\widetilde{\mathcal M}.
\]
Taking the supremum over all squares \(Q\subset\mathbb D\), and using
\(\widetilde v_\lambda=v_\lambda\) in \(\mathbb D\) gives
\[
\|\log|v_\lambda|\|_{\mathrm{BMO}(\mathbb D)}
\le
C\widetilde{\mathcal M}
\le
C_{\Omega,\varepsilon}
\left(1+\lambda^{\mu_\varepsilon(2,t_\Omega)}\right).
\]
This completes the proof.
\end{proof}

\begin{remark}
On a compact smooth Riemannian surface, propagation of
smallness estimates imply BMO-type bounds for \(\log |\varphi_\lambda|\) with the
natural eigenvalue scale \(\sqrt{\lambda}\) see \cite{LM18} and also the related earlier estimates of
Donnelly and Fefferman~\cite{DF90}. Proposition~\ref{prop:BMO-Lipschitz-domain} gives an analogue for
Dirichlet eigenfunctions in simply connected planar Lipschitz domains, one can easily generalize the argument to multiply connected domains. The loss is expressed through the integrability exponent of the derivative of
the conformal map. Natural question is whether this two-dimensional result holds in higher dimensions.
\end{remark}

\section{Proof of power weight Carleman inequalities}
\label{s:Carleman}
\subsection{Preliminary remarks, logarithmic coordinates, and $L^2$-estimates}
We now prove the Carleman estimate stated in Theorem~\ref{Carlpq}.
The proof follows the ideas of Jerison in \cite{Jer86} and is based on
the factorization of the Laplacian in polar coordinates and on spectral projection
estimates on the sphere by Sogge \cite{Sog86}. The aim of this section is  to keep track of the precise power of the parameter
\(\tau\) in a range of \(L^p\)-\(L^q\) Carleman inequalities.

We pass to logarithmic polar coordinates
\[
    x=e^\ell\omega,\qquad \omega\in S^{d-1},\qquad \ell=\log |x|.
\]
Then
\(
    |x|^{-d}\,dx=d\ell\,d\omega
\)
and
\begin{equation}
    e^{2\ell}\Delta
    =
    \partial_\ell^2+(d-2)\partial_\ell+\Delta_\omega ,
    \label{laplace}
\end{equation}
where \(\Delta_\omega\) is the Laplace--Beltrami operator on \(S^{d-1}\).
In what follows we write
\(
    \|\cdot\|_p=\|\cdot\|_{L^p(d\ell d\omega)}.
\)

Let \(E_k\) be the eigenspace of \(-\Delta_\omega\) corresponding to the eigenvalue
\(
    k(k+d-2),\) \(k=0,1,2,\ldots,
\)
and let \(P_k\) be the orthogonal projection onto \(E_k\). 

Suppose that $w\in C^\infty(\mathbb S^{d-1})$ and 
\(  w=\sum_{k\ge0}w_k,\ w_k=P_kw,
\)
then
\[
    -\Delta_\omega w
    =
    \sum_{k\ge0} k(k+d-2)w_k .
\]
In particular,
\[
    \|\Delta_\omega w\|^2_{L^2(d\ell d\omega)}
    =
    \sum_{k\ge0} k^2(k+d-2)^2
    \|w_k\|^2_{L^2(d\ell d\omega)}.
\]
If \(\Omega_1,\ldots,\Omega_d\) denote the tangential projections of the standard basis vector fields onto the
sphere, then
\begin{equation}
    \sum_{j=1}^d
    \|\Omega_j w\|^2_{L^2(d\ell d\omega)}
    =
    \sum_{k\ge0}
    k(k+d-2)
    \|w_k\|^2_{L^2(d\ell d\omega)}.
    \label{lll}
\end{equation}

Set
\[
\mathcal N=\left(-\Delta_\omega+\frac{(d-2)^2}{4}\right)^{1/2},
\qquad
L^\pm=\partial_\ell+\frac{d-2}{2}\pm\mathcal N.
\]
Since
\(\mathcal NP_k=(k+(d-2)/2)P_k\), the operators \(L^+\) and
\(L^-\) commute and
\[
L^+L^-=L^-L^+
=\partial_\ell^2+(d-2)\partial_\ell+\Delta_\omega
=e^{2\ell}\Delta.
\]

We next present the straightforward \(L^2\)-estimates for the first-order factors
\(L^+\) and \(L^-\). These estimates can be  obtained using integration by parts argument.  For example, (\ref{eq:Lplus-estimate}) and (\ref{eq:Lminus-estimate}) can be shown using Lemma 1 and Lemma 2 in \cite{DZ18} by replacing the weight functions $e^{-\tau(l+\ln l^2)}$ by $e^{-\tau l}$.
Hence, we skip the computations.

\begin{lemma}
\label{lem:first-order-factors}
There exists constant $C=C(d)$ such that the following inequalities hold. Let \(w\in C^\infty_0(\mathbb R\times S^{d-1})\).
(i) For every \(\tau>0\),
\begin{equation}
    \tau\|e^{-\tau\ell}w\|_2
    +
    \|e^{-\tau\ell}\partial_\ell w\|_2
    +
    \sum_{j=1}^d
    \|e^{-\tau\ell}\Omega_jw\|_2
    \le C    \|e^{-\tau\ell}L^+w\|_2 .
    \label{eq:Lplus-estimate}
\end{equation}
\noindent
(ii) If $\tau>0$ is non-resonant
then
\begin{equation}
    \|e^{-\tau\ell}w\|_2
    \le
    C
    \|e^{-\tau\ell}L^-w\|_2 .
    \label{eq:Lminus-estimate}
\end{equation}
\end{lemma}

\subsection{Spherical multiplier estimate and the \(L^p-L^2\) estimate for \(L^-\)}
We now improve the \(L^2\)-estimate for \(L^-\) to an \(L^p\)-estimate. One of the crucial tools is Sogge's spectral projection estimate, which we also formulate for the finite sums of spherical harmonics. 

\begin{knlemma}
\label{lem:sogge-projection}
Let \(d\ge3\) and
\(\frac{2d}{d+2}\le p\le2\le q\le\frac{2d}{d-2}\).
In dimension \(d=2\), let \(1\le p\le2\le q\le\infty\).
Let \(I\subset\mathbb N_0\) be finite and let
\(\{a_k\}_{k\in I}\) be complex numbers satisfying
\(|a_k|\le1\). Then
\begin{equation}
\left\|\sum_{k\in I}a_kP_kf\right\|_{L^q(S^{d-1})}
\le C
\left(\sum_{k\in I}|a_k|(1+k)^{(d-2)/d}\right)^{
\frac d2(\frac1p-\frac1q)}
\|f\|_{L^p(S^{d-1})}.
\label{eq:sogge-finite-sum}
\end{equation}
The constant \(C\) depends only on \(d,p\), and \(q\).
In particular,
\begin{equation}
\|P_kf\|_{L^q(S^{d-1})}
\le C(1+k)^{\frac{d-2}{2}(\frac1p-\frac1q)}
\|f\|_{L^p(S^{d-1})}.
\label{eq:sogge-single-projection}
\end{equation}
\end{knlemma}

The estimate \eqref{eq:sogge-finite-sum} is a standard consequence of Sogge's spectral
cluster estimates.  A more general estimate is also proved in \cite{D20}. 
Using the above Lemma, following the scheme of Jerison's proof in \cite{Jer86}, and carefully tracking the power of $\tau$,  one obtains 
\begin{lemma}
\label{lem:Lminus-Lp-L2}
Let
\(
    \frac{2d}{d+2}\le p\le2,
\) for $d\ge 3$ and $1<p\le 2$ for $d=2$.
Assume that \(\tau>0\) is non-resonant.
Then, for every \(w\in C^\infty_0(\mathbb R\times S^{d-1})\),
\begin{equation}
    \|e^{-\tau\ell}w\|_{L^2(d\ell d\omega)}
    \le
    C\tau^{\beta_0}
    \|e^{-\tau\ell}L^-w\|_{L^p(d\ell d\omega)},
    \qquad
    \beta_0=\frac{(d-2)(2-p)}{4p}.
    \label{eq:Lminus-Lp-L2}
\end{equation}
\end{lemma}

 For \(\tau>0\), define the conjugated operators
\(  L^\pm_\tau=e^{-\tau\ell}L^\pm e^{\tau\ell}.
\)
Then
\[
    L^+_\tau
    =
    \partial_\ell+\tau+\frac{d-2}{2}+\mathcal N, \quad 
    L^-_\tau
    =
    \partial_\ell+\tau+\frac{d-2}{2}-\mathcal N .
\]
Set
\(
    \widetilde{w}=e^{-\tau\ell}w.
\)
Then
\(
    e^{-\tau\ell}L^\pm w=L^\pm_\tau \widetilde{w}.
\)

\begin{proof}[Proof of Lemma \ref{lem:Lminus-Lp-L2}]
For \(p=2\), this is precisely the second estimate in
Lemma~\ref{lem:first-order-factors}. 
 We therefore assume that \(p<2\) and set
\[
\alpha=\frac d2\left(\frac1p-\frac12\right),\qquad
a=\frac{d-2}{d}.
\]
Then
$\beta_0=a\alpha$.
Set
$g=L^-_\tau\ww=e^{-\tau\ell}L^-w.$
Decompose \(\ww\) into spherical harmonics,
\(\ww=\sum_{k\ge0}\ww_k\), where \(\ww_k=P_k\ww\). Then
\[
P_kg=(\partial_\ell+\tau-k)\ww_k.
\]
Solving this ODE, we obtain
\begin{equation*}
\ww_k(\ell,\omega)
=
\int_{\mathbb R}
S_k(\ell,s)P_kg(s,\omega)\,ds,
\quad
{\text{where}}\quad
S_k(\ell,s)
=
\begin{cases}
e^{-(\tau-k)(\ell-s)}
\mathbf 1_{\{s<\ell\}},
& k<\tau,\\[2mm]
-e^{-(k-\tau)(s-\ell)}
\mathbf 1_{\{s>\ell\}},
& k>\tau.
\end{cases}
\end{equation*}
In particular,
\(
|S_k(\ell,s)|
\le e^{-|\tau-k||\ell-s|}
\le1.
\)

Fix \(\ell,s\),  applying
Lemma~\ref{lem:sogge-projection} with \(q=2\) to the partial sums, gives
\begin{equation}\label{eq:normsK}
\left\|
\sum_{k=0} S_k(\ell,s)P_kf
\right\|_{L^2(S^{d-1})}
\le
C K_{\tau,p}(\ell-s)
\|f\|_{L^p(S^{d-1})},
\end{equation}
where
\[
K_{\tau,p}(h)
=
\left(
\sum_{k\ge0}
e^{-|\tau-k||h|}(1+k)^a
\right)^\alpha.
\]
We claim that
\begin{equation}\label{eq:Kest}
K_{\tau,p}(h)
\le
C(1+\tau)^{\beta_0}
\begin{cases}
|h|^{-\sigma},&0<|h|\le1,\\
e^{-c|h|},&|h|\ge1,
\end{cases}
\qquad
\sigma=(1+a)\alpha.
\end{equation}

Suppose first that \(0<|h|\le1\). Since
$1+k\le C\bigl(1+\tau+|k-\tau|\bigr),$
we have
\[
\sum_{k\ge0}e^{-|\tau-k||h|}(1+k)^a
\le
C(1+\tau)^a
\sum_{k\ge0}e^{-|\tau-k||h|}+
C\sum_{k\ge0}
e^{-|\tau-k||h|}|k-\tau|^a.
\]
Write \(\tau=m+\theta\), where \(m\in\mathbb N_0\). By
non-resonance condition, \(1/3<\theta<2/3\). 
Then
\[
\sum_{k\ge0}e^{-|\tau-k||h|}
\le
C\sum_{j\ge0}e^{-(j+1/3)|h|}
\le C|h|^{-1},
\]
and, by comparison with the corresponding integral,
\[
\sum_{k\ge0}
e^{-|\tau-k||h|}|k-\tau|^a
\le
C\sum_{j\ge0}
e^{-(j+1/3)|h|}(j+1)^a
\le C|h|^{-1-a}.
\]
It follows that
\[
\sum_{k\ge0}e^{-|\tau-k||h|}(1+k)^a
\le
C(1+\tau)^a|h|^{-1}
+C|h|^{-1-a}
\le
C(1+\tau)^a|h|^{-1-a}.
\]
Raising this inequality to the power \(\alpha\), and using
\(a\alpha=\beta_0\) and \((1+a)\alpha=\sigma\), gives
\[
K_{\tau,p}(h)
\le
C(1+\tau)^{\beta_0}|h|^{-\sigma}.
\]

If \(|h|\ge1\), the same decomposition into the two sequences gives
\[
\sum_{k\ge0}e^{-|\tau-k||h|}(1+k)^a
\le
C\sum_{j\ge0}
e^{-(j+1/3)|h|}
\bigl((1+\tau)^a+(j+1)^a\bigr)\le
C(1+\tau)^a e^{-c|h|}.
\]
This proves the claim.

Let $G(s)=\|g(s,\cdot)\|_{L^p(S^{d-1})}.$
Then, applying  Minkowski's inequality to the integral representation of $\ww_k$, and using \eqref{eq:normsK} and \eqref{eq:Kest},
we obtain
\[
\|\ww(\ell,\cdot)\|_{L^2(S^{d-1})}
\le
C(1+\tau)^{\beta_0}
\int_{\mathbb R}\kappa_p(\ell-s)G(s)\,ds,
\]
where
\(
\kappa_p(h)
=
|h|^{-\sigma}\mathbf 1_{\{|h|\le1\}}
+
e^{-c\alpha|h|}\mathbf 1_{\{|h|>1\}}.
\)

Assume first that $d\left(\frac1p-\frac12\right)<1$.
This includes $\frac{2d}{d+2}<p<2$ when $d\ge3,$
and the whole range \(1<p<2\) when \(d=2\).
Let \(r\) be defined by
\(
\frac1r+\frac1p=\frac32.\)
Then
$\sigma r<1$ and 
we have \(\kappa_p\in L^r(\mathbb R)\). Young's convolution
inequality therefore gives
\[
\|\ww\|_{L^2(d\ell d\omega)}
\le
C(1+\tau)^{\beta_0}
\|g\|_{L^p(d\ell d\omega)}.
\]

It remains to consider the endpoint
$p=p_0=\frac{2d}{d+2}$ for $d\ge 3$.
In this case
$
\sigma=\frac{d-1}{d},
$
and the preceding kernel estimate gives
\[
\kappa_{p_0}(h)
\le C|h|^{-(d-1)/d}
=C|h|^{-1+1/d}.
\]
Since
\(
\frac1{p_0}-\frac12=\frac1d,
\)
the one-dimensional Hardy--Littlewood--Sobolev inequality gives
\[
\|\ww\|_{L^2(d\ell d\omega)}
\le
C(1+\tau)^{\beta_0}
\|g\|_{L^{p_0}(d\ell d\omega)}.
\]

Finally,
\(
(1+\tau)^{\beta_0}\le C\tau^{\beta_0}.
\)
Recalling that
\(
\ww=e^{-\tau\ell}w\) and 
\(g=e^{-\tau\ell}L^-w.
\)
we obtain
the required estimate.
\end{proof}

\subsection{Carleman inequalities for Laplacian}
Thanks to the Carleman estimates for $L^+$ and $L^-$, we are able to show the $L^p-L^2$ type estimates for Laplace operators.
\begin{proposition}
    Let $\frac{2d}{d+2} \leq p \le 2$ for $d\ge 3$ and $1<p\le 2$ for $d=2$. 
There exists a constant $C = C\pr{d,p}$  such that for any $v\in C^{\iny}_{0}\pr{\mathbb R^d\setminus\{0\} }$ and non-resonant $\tau>1$, one has
\begin{equation}
\tau^{\be_1} \norm{  |x|^{-\tau}v}_{L^2(|x|^{-d}dx)}
+ \tau^{\be_1-1} \norm{   |x|^{1-\tau} \gr v}_{L^2(|x|^{-d}dx)} 
\le C \norm{  |x|^{2-\tau} \LP v}_{L^p(|x|^{-d} dx)} ,
\label{kao-1}
\end{equation}
where $\be_1 =1- \frac{\pr{d-2}\pr{2-p}}{4p}$. 
\label{pro2}
\end{proposition}

\begin{proof}
Let $v\in C^{\infty}_{0}\pr{\mathbb R^d\setminus\set{0} }$ and set \(w(\ell,\omega)=v(e^\ell\omega)\).
Applying inequality \eqref{eq:Lplus-estimate} to \(w\) and then \eqref{eq:Lminus-Lp-L2}   to $L^+w$, we obtain
\begin{align*}
\tau \norm{ e^{-\tau \ell}w}_{L^2(d\ell d\omega )}
&+\norm{e^{-\tau \ell} \partial_\ell w}_{L^2(d\ell d\omega )}
+\sum_{i=1}^d \norm{e^{-\tau \ell} \Omega_i w }_{L^2(d\ell d\omega )}   \\
&\leq C\norm{e^{-\tau \ell} L^+w}_{L^2(d\ell d\omega )} 
\le C\tau^{ \frac{\pr{d-2}\pr{2-p}}{4p}}\norm{
e^{-\tau \ell} L^- L^+w}_{L^p(d\ell d\omega )}.
\end{align*}
Returning to Euclidean coordinates and noticing that 
\[|x||\nabla v|\le C\left(|\partial_\ell w|+\sum_{i=1}^d|\Omega_i w|\right),\]
we obtain the required inequality.
\end{proof}

Now we formulate a $L^p-L^q$ type Carleman estimate for the operator $\LP$.

\begin{proposition}
Let $\frac{2d}{d+2} \le p \le 2 \le q \le \frac{2d}{d-2}$ be such that $\frac 1 p - \frac 1 q \leq \frac 2 d$ for $d\ge 3$ and $1<p\le 2\le q\le \infty$ for $d=2$.
There exists a constant $C = C\pr{d,p,q}$ such that for any $v\in C^{\infty}_0 \pr{\mathbb R^d\setminus\set{0} }$ and non-resonant $\tau>1$ it holds that
\begin{equation}
\tau^{\hat{\be}_0}\norm{  |x|^{-\tau}v}_{L^q(|x|^{-d}dx)}
\le  C  \norm{|x|^{2-\tau} \LP v}_{L^p(|x|^{-d} dx)},
\label{key}
\end{equation}
where $\hat{\be}_0 = 1 - \frac{d}{2}\pr{\frac 1 p - \frac 1 q}$. 
\label{CarL+L-pq}
\end{proposition}

\begin{proof}
Let \(w(\ell,\omega)=v(e^\ell\omega)\), set
\(\ww=e^{-\tau\ell}w\), and define
\(F=e^{-\tau\ell}L^+L^-w\). Then
\(
L^+_\tau L^-_\tau\ww=F.
\) Set \(\delta=1/p-1/q\), \(\alpha=d\delta/2\),
\(\widehat\beta_0=1-\alpha\), and
\(\gamma=(d-2)\delta/2\). The case \(\delta=0\) follows from
Proposition~\ref{pro2}, so we assume that \(0<\delta<2/d\). 
As in the proof of Lemma~\ref{lem:Lminus-Lp-L2}, we decompose
\[
\ww=P^-_\tau\ww+P^+_\tau\ww,\qquad
P^-_\tau\ww=\sum_{k<2\tau}P_k\ww,\qquad
P^+_\tau\ww=\sum_{k\ge2\tau}P_k\ww.
\]

We first estimate \(P^-_\tau\ww\). Let
\(g=L^-_\tau\ww\). Since \(L^+_\tau g=F\), we have
\[
P_kg(y,\omega)
=\int_{-\infty}^y e^{-a_k(y-s)}P_kF(s,\omega)\,ds,
\qquad a_k=\tau+k+d-2.
\]
Using the kernel \(S_k\) from the proof of
Lemma~\ref{lem:Lminus-Lp-L2}, we obtain
\begin{equation}
P^-_\tau\ww(\ell,\omega)
=\int_{\mathbb R}\int_{-\infty}^y e^{-\tau(y-s)}
\sum_{k<2\tau}c_k(\ell,y,s)P_kF(s,\omega)\,ds\,dy,
\label{eq:low-frequency-representation}
\end{equation}
where
\(
c_k(\ell,y,s)
=S_k(\ell,y)e^{-(k+d-2)(y-s)}.
\)
Since \(s<y\), the estimate for \(S_k\) gives
\[
|c_k(\ell,y,s)|
\le e^{-|\tau-k||\ell-y|}\le1.
\]

For fixed \(\ell,y,s\), the \(p\)-\(q\) version of
Lemma~\ref{lem:sogge-projection} therefore gives
\begin{equation}
\left\|
\sum_{k<2\tau}c_k(\ell,y,s)P_kF(s,\cdot)
\right\|_{L^q(S^{d-1})}
\le
C\tau^\gamma H_\tau(\ell-y)
\|F(s,\cdot)\|_{L^p(S^{d-1})},
\label{eq:low-angular-multiplier}
\end{equation}
where
\[
H_\tau(h)
=
\left(
\sum_{k<2\tau}e^{-|\tau-k||h|}
\right)^\alpha.
\]
Indeed, \(1+k\le C\tau\) for \(k<2\tau\), and hence the factor
coming from the spherical multiplier estimate is
\[
\tau^{\frac{d-2}{d}\alpha}
=\tau^{\frac{d-2}{2}\delta}
=\tau^\gamma.
\]

Let \(G(s)=\|F(s,\cdot)\|_{L^p(S^{d-1})}\). Applying
\eqref{eq:low-angular-multiplier} in
\eqref{eq:low-frequency-representation} and using Minkowski's
inequality, we obtain
\begin{equation}
\|P^-_\tau\ww(\ell,\cdot)\|_{L^q(S^{d-1})}
\le
C\tau^\gamma
\int_{\mathbb R}\int_{-\infty}^y
H_\tau(\ell-y)e^{-\tau(y-s)}G(s)\,ds\,dy.
\label{eq:low-double-convolution}
\end{equation}

As in the proof of Lemma \ref{lem:Lminus-Lp-L2}, we have
\[
H_\tau(h)\le C|h|^{-\alpha}
\quad\text{for \(0<|h|\le1\),}
\qquad
H_\tau(h)\le Ce^{-c\alpha|h|}
\quad\text{for \(|h|\ge1\).}
\]
Since \(\alpha<1\), these estimates imply
\(\|H_\tau\|_{L^1(\mathbb R)}\le C\). Moreover, the sum defining
\(H_\tau\) contains at most \(C\tau\) terms, and hence
\(
\|H_\tau\|_{L^\infty(\mathbb R)}\le C\tau^\alpha.
\)

Set
\[
E_\tau(h)=e^{-\tau h}\mathbf 1_{\{h>0\}},
\qquad K_\tau=H_\tau*E_\tau.
\]
Then \(\|E_\tau\|_{L^1}=\tau^{-1}\), and therefore
\[
\|K_\tau\|_{L^1}\le C\tau^{-1},
\qquad
\|K_\tau\|_{L^\infty}\le C\tau^{\alpha-1}.
\]
Let \(r\) be defined by \(1/r=1-\delta\), so that
\(1+1/q=1/p+1/r\). Interpolation between the preceding
\(L^1\)- and \(L^\infty\)-bounds gives
\[
\|K_\tau\|_{L^r}
\le
C\tau^{-1/r+(\alpha-1)(1-1/r)}
\le C\tau^{-1/r},
\]
where we used \(\alpha<1\) and \(\tau>1\).

The right-hand side of \eqref{eq:low-double-convolution} is
\(C\tau^\gamma(K_\tau*G)(\ell)\). Young's inequality now gives
\[
\|P^-_\tau\ww\|_{L^q(d\ell d\omega)}
\le
C\tau^{\gamma-1/r}
\|F\|_{L^p(d\ell d\omega)}.
\]
Finally,
\[
\gamma-\frac1r
=
\frac{d-2}{2}\delta-(1-\delta)
=
\frac d2\delta-1
=
-\widehat\beta_0.
\]
Thus
\begin{equation}
\|P^-_\tau\ww\|_{L^q(d\ell d\omega)}
\le
C\tau^{-\widehat\beta_0}
\|F\|_{L^p(d\ell d\omega)}.
\label{eq:low-frequency-bound}
\end{equation}

We now estimate \(P^+_\tau\ww\). For \(k\ge2\tau\), we have
\(b_k=\tau-k<0\). Using the representation of \(g_k=P_kg\) above
and inverting \(\partial_\ell+b_k\), we obtain
\[
P_k\ww(\ell,\omega)
=\int_{\mathbb R}K_k^+(\ell,s)P_kF(s,\omega)\,ds,
\]
where
\[
K_k^+(\ell,s)
=-\int_{\max\{\ell,s\}}^\infty
e^{b_k(y-\ell)}e^{-a_k(y-s)}\,dy.
\]
Since \(a_k-b_k=2k+d-2\), this gives that
\[
K^+_k(\ell,s)
=-\frac{1}{2k+d-2}
\begin{cases}
e^{-a_k(\ell-s)},&s<\ell,\\
e^{-(k-\tau)(s-\ell)},&s\ge\ell,
\end{cases}
\]
where \(a_k=\tau+k+d-2\). Since \(k-\tau\ge k/2\) and
\(a_k\ge k\), it follows that
\[
|K^+_k(\ell,s)|\le Ck^{-1}e^{-ck|\ell-s|}.
\]

We decompose \(P^+_\tau\ww\) into dyadic blocks. For \(N\ge1\), set
\[
\tau_N=2^N\tau,\qquad
\ww^{(N)}=\sum_{\tau_N\le k<2\tau_N}P_k\ww,\qquad 
P^+_\tau\ww=\sum_{N\ge1}\ww^{(N)}.
\]
Then
\[
\ww^{(N)}(\ell,\omega)
=\int_{\mathbb R}
\sum_{\tau_N\le k<2\tau_N}
K^+_k(\ell,s)P_kF(s,\omega)\,ds.
\]

Recall that
\[
\gamma=\frac{d-2}{2}\delta
\]
and \(G(s)=\|F(s,\cdot)\|_{L^p(S^{d-1})}\). By the
single-projection estimate \eqref{eq:sogge-single-projection} and
the triangle inequality,
\[
\left\|
\sum_{\tau_N\le k<2\tau_N}
K^+_k(\ell,s)P_kF(s,\cdot)
\right\|_{L^q(S^{d-1})}
\le
C\sum_{\tau_N\le k<2\tau_N}
k^{\gamma-1}e^{-ck|\ell-s|}G(s).
\]
Since \(k\) is comparable to \(\tau_N\) on this block and the block
contains at most \(C\tau_N\) indices, the last expression is bounded
by
\(
C\tau_N^\gamma e^{-c\tau_N|\ell-s|}G(s).
\)
Consequently,
\[
\|\ww^{(N)}(\ell,\cdot)\|_{L^q(S^{d-1})}
\le
C\tau_N^\gamma
\int_{\mathbb R}e^{-c\tau_N|\ell-s|}G(s)\,ds.
\]

Recall that \(1/r=1-\delta\), so that
\(1+1/q=1/p+1/r\). Since
\(
\|e^{-c\tau_N|\cdot|}\|_{L^r(\mathbb R)}
\le C\tau_N^{-1/r},
\)
applying Young's inequality gives that
\[
\|\ww^{(N)}\|_{L^q(d\ell d\omega)}
\le
C\tau_N^{\gamma-1/r}
\|F\|_{L^p(d\ell d\omega)}.
\]
Moreover,
\[
\gamma-\frac1r
=
\frac{d-2}{2}\delta-(1-\delta)
=
\frac d2\delta-1
=
-\widehat\beta_0.
\]
Hence
\[
\|\ww^{(N)}\|_{L^q(d\ell d\omega)}
\le
C(2^N\tau)^{-\widehat\beta_0}
\|F\|_{L^p(d\ell d\omega)}.
\]

Since \(\widehat\beta_0>0\), summing over \(N\ge1\) yields that
\begin{equation}
\|P^+_\tau\ww\|_{L^q(d\ell d\omega)}
\le
C\tau^{-\widehat\beta_0}
\|F\|_{L^p(d\ell d\omega)}.
\label{eq:high-frequency-bound}
\end{equation}

Combining \eqref{eq:low-frequency-bound} and
\eqref{eq:high-frequency-bound}, we obtain
\[
\|\ww\|_{L^q(d\ell d\omega)}
\le
C\tau^{-\widehat\beta_0}
\|F\|_{L^p(d\ell d\omega)}.
\]
Since \(\ww=e^{-\tau\ell}w\),
\(w(\ell,\omega)=v(e^\ell\omega)\), and
\(L^+L^-w=e^{2\ell}\Delta v\), this is equivalent to
\[
\tau^{\widehat\beta_0}
\bigl\||x|^{-\tau}v\bigr\|_{L^q(|x|^{-d}\,dx)}
\le
C
\bigl\||x|^{2-\tau}\Delta v\bigr\|_{L^p(|x|^{-d}\,dx)}.
\]
This proves the proposition when
\(1/p-1/q<2/d\). For $d\ge 3$, the endpoint
\(1/p-1/q=2/d\) is the classical Jerison--Kenig estimate \cite{JK85}.
\end{proof}

To prove Theorem \ref{Carlpq}, we interpolate between Proposition \ref{pro2} at $q=2$ and Proposition \ref{CarL+L-pq} at $q = \frac{2d}{d-2}.$ 

\begin{proof}[Proof of Theorem \ref{Carlpq}]
To show Theorem \ref{Carlpq}, it suffices to show that
\[
\tau^\be \norm{ |x|^{-\tau} v}_{L^q(|x|^{-d}dx)}
\le  C\norm{ |x|^{2-\tau} \LP v}_{L^p(|x|^{-d} dx)}.
\]
For any $q \in \brac{2, \frac{2d}{d-2}}$, we express $q = 2 \lambda + \frac{2d}{d-2}\pr{1 - \lambda}$ for some $\lambda \in \brac{0,1}$. Thus, $\lambda=\frac{2d-(d-2)q}{4}.$
Combining H\"older's inequality, the estimates \eqref{kao-1} in Proposition \ref{pro2},  and the estimates \eqref{key} with $q = \frac{2d}{d-2}$ in Proposition \ref{CarL+L-pq}, we have
\begin{multline*}
    \norm{|x|^{-\tau}v}_{L^q(|x|^{-d}dx)} 
\le \norm{ |x|^{-\tau} v}_{L^2(|x|^{-d}dx)}^{\frac{2 \lambda}{q}} 
\norm{ |x|^{-\tau}v}_{L^{\frac{2d}{d-2}}(|x|^{-d}dx)}^{\frac{2d\pr{1-\lambda}}{\pr{d-2}q}} \\
\le \pr{C \tau^{-\be_1} \norm{ |x|^{2-\tau }  \LP v}_{L^p(|x|^{-d} dx)} }^{\frac{2 \lambda}{q}}
\pr{C \tau^{-\hat{\be}_0} \norm{ |x|^{2-\tau}  \LP v}_{L^p(|x|^{-d} dx)} }^{\frac{2d\pr{1-\lambda}}{\pr{d-2}q}} \\
= C \tau^{-\be_1\frac{2 \lambda}{q}-\hat{\be}_0{\frac{2d\pr{1-\lambda}}{\pr{d-2}q}}}  \norm{ |x|^{2-\tau} \LP v}_{L^p(|x|^{-d} dx)}.
\end{multline*}
From the definition of $\beta_1$ in Proposition \ref{pro2} and $\hat{\beta}_0$ in Proposition \ref{CarL+L-pq}, we  get  
\[\be_1\frac{2 \lambda}{q} + \hat{\be}_0{\frac{2d\pr{1-\lambda}}{\pr{d-2}q}} =\hat{\be}_0+(\beta_1-\hat{\be}_0) \frac{2 \lambda}{q}=\beta=\frac{2+d}{4}+\frac{2-d}{2p}+\frac{d}{p}\left(\frac{1}{q}-\frac{1}{2}\right).\] For $d=2$, we replace $L^{2d/(d-2)}$ by $L^\infty$. This concludes the proof of Theorem \ref{Carlpq}.
\end{proof}

 \appendix
 
 \section{Local boundedness and oscillation lemmas}
 \label{s:local-estimates}
In this appendix we record the local boundedness and oscillation estimates used in
Section~\ref{s:Remez}. The solution and the potential may be complex-valued, so we do
not use the maximum principle, Harnack's inequality, or the real-valued oscillation
\(\sup u-\inf u\). For a complex-valued function \(u\) and a set \(E\), we write
\(\operatorname{osc}^{\mathbb C}_E u:=\sup_{x,y\in E}|u(x)-u(y)|\).

\begin{lemma}
\label{lem:L2-Linfty-local}
Let \(d\ge2\), \(t>d/2\), and let \(u\in W^{1,2}(B_R)\)
be a complex-valued weak solution of \eqref{ePDE} in \(B_R\), where
\(V\in L^t(B_R)\). Then
\begin{equation}
\|u\|_{L^\infty(B_{R/2})}
\le CR^{-d/2}
\left(1+\left(R^{2-d/t}\|V\|_{L^t(B_R)}\right)^{\frac{2t}{2t-d}}\right)^{d/4}
\|u\|_{L^2(B_R)},
\label{eq:local-L2-Linfty}
\end{equation}
where the constant \(C\) depends only on \(d\) and \(t\). 
\end{lemma}

\begin{proof}
By scaling, it is enough to consider \(R=1\). Set
\(v=|u|\), \(A=\|V\|_{L^t(B_1)}\), and \(\kappa=2t/(2t-d)\), with \(\kappa=1\) when
\(t=\infty\). Since \(t>d/2\), Sobolev's inequality implies that \(Vu\in L^1(B_1)\). Therefore 
Kato's inequality~\cite{Kato78} gives
\(-\Delta v\le |V|v\) in the sense of distributions.

Let \(0<r<s\le1\), and choose \(\eta\in C^\infty_0(B_s)\) such that \(\eta=1\) on
\(B_r\) and \(|\nabla\eta|\le C(s-r)^{-1}\). Testing the subsolution inequality by a
truncated version of \(\eta^2v^{2m-1}\), and then removing the truncation gives
\begin{equation}
\|\nabla(\eta v^m)\|_{L^2(B_s)}^2
\le Cm^2(s-r)^{-2}\|v^m\|_{L^2(B_s)}^2
   +Cm\int_{B_s}|V|\eta^2v^{2m}
\label{eq:power-caccioppoli}
\end{equation}
 for
all \(m\ge1\).
Assume first that \(t<\infty\). For \(f=\eta v^m\), H\"older's inequality and the
Gagliardo--Nirenberg--Sobolev inequality give
\[
\int_{B_s}|V|f^2
\le A\|f\|_{L^{2t/(t-1)}(B_s)}^2
\le CA\|\nabla f\|_{L^2(B_s)}^{d/t}\|f\|_{L^2(B_s)}^{2-d/t}.
\]
Since \(d/(2t)<1\), Young's inequality yields that
\[
CmA\|\nabla f\|_{L^2}^{d/t}\|f\|_{L^2}^{2-d/t}
\le \frac12\|\nabla f\|_{L^2}^2
   +Cm^\kappa A^\kappa\|f\|_{L^2}^2.
\]
For \(t=\infty\), the same conclusion follows by bounding the potential term by
\(A\|f\|_{L^2}^2\). Thus \eqref{eq:power-caccioppoli} implies that 
\begin{equation}
\|\nabla(\eta v^m)\|_{L^2(B_s)}^2
\le C\left(m^2(s-r)^{-2}+m^\kappa A^\kappa\right)
\|v^m\|_{L^2(B_s)}^2.
\label{eq:moser-energy}
\end{equation}

Suppose first that \(d\ge3\), and set \(\chi=d/(d-2)\). The Sobolev inequality and
\eqref{eq:moser-energy} give
\[
\|v\|_{L^{2m\chi}(B_r)}
\le C^{1/(2m)}
\left(m^2(s-r)^{-2}+m^\kappa A^\kappa\right)^{1/(2m)}
\|v\|_{L^{2m}(B_s)}.
\]
We iterate this estimate with \(m_j=\chi^j\) and
\(r_j=\frac12+2^{-j-1}\). Since \(\sum_{j\ge0}(2m_j)^{-1}=d/4\), while the products
containing powers of \(m_j\) and \((r_j-r_{j+1})^{-1}\) converge, we obtain
\[
\|v\|_{L^\infty(B_{1/2})}
\le C(1+A^\kappa)^{d/4}\|v\|_{L^2(B_1)}.
\]

If \(d=2\), the following inequality hold 
$|f\|_{L^4}^2\le C\|\nabla f\|_{L^2}\|f\|_{L^2}$, \cite{L59}. Combined  with
\eqref{eq:moser-energy}, it gives
\[
\|v\|_{L^{4m}(B_r)}
\le C^{1/(2m)}
\left(m^2(s-r)^{-2}+m^\kappa A^\kappa\right)^{1/(4m)}
\|v\|_{L^{2m}(B_s)}.
\]
We now take \(m_j=2^j\) and the same radii \(r_j\). Since
\(\sum_{j\ge0}(4m_j)^{-1}=1/2=d/4\), the iteration gives the same estimate. Scaling
back to \(B_R(x_0)\) proves \eqref{eq:local-L2-Linfty}.
\end{proof}

The next lemma shows, in particular, that a weak solution is H\"older
continuous.

\begin{lemma}
\label{lem:complex-oscillation}
Let \(d\ge2\), \(d/2<t\le\infty\), and let
\(u\in W^{1,2}_{\mathrm{loc}}(B_{2R})\) be a complex-valued weak solution of
\(\Delta u+V(x)u=0\) in \(B_{2R}\), where
\(V\in L^t(B_{2R})\). Define
\(\alpha=\frac12\min\{1,2-d/t\}\) when \(t<\infty\), and \(\alpha=1/2\) when
\(t=\infty\). Then, for every \(0<r\le R\),
\begin{equation}
\operatorname{osc}^{\mathbb C}_{B_r}u
\le C\left(\frac rR\right)^\alpha
\left(1+R^{2-d/t}\|V\|_{L^t(B_{2R})}\right)
\|u\|_{L^\infty(B_{2R})},
\label{eq:scaled-complex-oscillation}
\end{equation}
where \(C=C(d,t)\). For \(t=\infty\), the potential term is understood as
\(R^2\|V\|_{L^\infty(B_{2R})}\).
\end{lemma}

\begin{proof}
By scaling, it is enough to take \(R=1\). By
Lemma~\ref{lem:L2-Linfty-local}, the solution is locally bounded. Assume first that
\(t<\infty\). Since \(\Delta u=-Vu\), the interior \(W^{2,t}\)-estimate for the
Laplacian~\cite[Theorem~9.11]{GT77} gives
\[
\|u\|_{W^{2,t}(B_1)}
\le C\left(\|Vu\|_{L^t(B_{3/2})}
          +\|u\|_{L^t(B_{3/2})}\right)
\le C\left(1+\|V\|_{L^t(B_2)}\right)
\|u\|_{L^\infty(B_2)}.
\]
Our choice of \(\alpha\) satisfies
\(0<\alpha<\min\{1,2-d/t\}\). The Sobolev--Morrey embedding therefore implies
\[
|u(y)-u(z)|
\le C|y-z|^\alpha
\left(1+\|V\|_{L^t(B_2)}\right)
\|u\|_{L^\infty(B_2)},\qquad y,z\in B_1.
\]
Taking the supremum over \(y,z\in B_r\) proves the estimate for \(t<\infty\).

If \(t=\infty\), apply the interior \(W^{2,p}\)-estimate with \(p=2d\). Since
\(W^{2,2d}(B_1)\) embeds into \(C^{1/2}(B_1)\), the same argument gives
\[
\operatorname{osc}^{\mathbb C}_{B_1}u
\le Cr^{1/2}
\left(1+\|V\|_{L^\infty(B_2)}\right)
\|u\|_{L^\infty(B_2)}.
\]
Scaling back proves \eqref{eq:scaled-complex-oscillation}.
\end{proof}

\if false

In this appendix we record a local oscillation estimate which is used in Section~5.
The point is that the argument must apply to complex-valued solutions and complex-valued
potentials. Therefore we do not use the maximum principle, Harnack's inequality, or the
real-valued oscillation \(\sup u-\inf u\).

For a complex-valued function \(u\) and a measurable set \(E\), we define
\[
    \operatorname{osc}^{\mathbb C}_{E}u
    :=
    \sup_{x,y\in E}|u(x)-u(y)|.
\]

We shall use the following regularity lemma.

\begin{lemma}
\label{lem:complex-local-regularity}
Let \(d\ge3\), let \(t>d/2\), and let \(V\in L^t_{\mathrm{loc}}(\Omega;\mathbb C)\).
Assume that \(u\in W^{1,2}_{\mathrm{loc}}(\Omega;\mathbb C)\) is a weak solution of
\[
    \Delta u+V(x)u=0
    \qquad\text{in }\Omega .
\]
Then \(u\in L^\infty_{\mathrm{loc}}(\Omega)\). Moreover, if \(t<\infty\), then
\[
    u\in W^{2,t}_{\mathrm{loc}}(\Omega),
\]
while if \(t=\infty\), then
\[
    u\in W^{2,p}_{\mathrm{loc}}(\Omega)
    \qquad\text{for every }1<p<\infty.
\]
Consequently \(u\) is locally Hölder continuous.
\end{lemma}

\begin{proof}
The local boundedness follows from the Brezis--Kato bootstrap for the Schrödinger
operators with singular complex potentials; see Brezis--Kato~\cite{BK79}.

For completeness, we indicate why it applies here. Since
\[
    -\Delta u=V u,
\]
we are in the Brezis--Kato setting since \(V\in L^t\) with \(t>d/2\). The Brezis--Kato result implies
\[
    u\in L^q_{\mathrm{loc}}(\Omega)
    \qquad\text{for every }1\le q<\infty .
\]
Choose \(q\) sufficiently large so that
\[
    \frac1s=\frac1t+\frac1q
    \qquad\text{satisfies}\qquad
    s>\frac d2 .
\]
Then \(Vu\in L^s_{\mathrm{loc}}\). By the interior Calderón--Zygmund estimate for the
Poisson equation, see Gilbarg--Trudinger~\cite[Corollary~9.10, Theorem~9.11]{GT77},
we get
\[
    u\in W^{2,s}_{\mathrm{loc}}(\Omega).
\]
Since \(s>d/2\), Sobolev--Morrey embedding gives \(u\in L^\infty_{\mathrm{loc}}(\Omega)\).

Now \(Vu\in L^t_{\mathrm{loc}}\) if \(t<\infty\), and another application of the interior
\(W^{2,t}\)-estimate gives \(u\in W^{2,t}_{\mathrm{loc}}(\Omega)\). If \(t=\infty\), then
\(Vu\in L^p_{\mathrm{loc}}\) for every finite \(p\), and hence
\(u\in W^{2,p}_{\mathrm{loc}}(\Omega)\) for every \(1<p<\infty\).

\end{proof}

We now give the supremum norm estimate for weak solutions.
\begin{lemma}
\label{lem:L2-Linfty-local}
Let \(d\ge2\), \(t>d/2\), and let
\(
    u\in W^{1,2}_{\mathrm{loc}}(B_R(x_0))
\)
be a weak solution of \eqref{ePDE}
in \(B_R(x_0)\), where \(V\in L^t(B_R(x_0))\) may be complex-valued. Then
\[
    \|u\|_{L^\infty(B_{R/2}(x_0))}
    \le
    C R^{-d/2}
    \left(
        1+
        \left(
            R^{2-d/t}\|V\|_{L^t(B_R(x_0))}
        \right)^{\frac{2t}{2t-d}}
    \right)^{d/4}
    \|u\|_{L^2(B_R(x_0))}.
\]
The constant \(C\) depends only on \(d\) and \(t\). For \(t=\infty\), the estimate is
interpreted as
\[
    \|u\|_{L^\infty(B_{R/2}(x_0))}
    \le
    C R^{-d/2}
    \left(
        1+R^2\|V\|_{L^\infty(B_R(x_0))}
    \right)^{d/4}
    \|u\|_{L^2(B_R(x_0))}.
\]
\end{lemma}

\begin{proof}
By scaling and translation it is enough to prove the estimate for \(B_R(x_0)=B_1\).
Let
\(
    A=\|V\|_{L^t(B_1)}.
\)
Since \(u\) and \(V\) may be complex-valued, we apply Kato's inequality, see \cite{Kato78} to \(|u|\). In
the sense of distributions,
\[
    -\Delta |u|
    \le
    |V|\,|u|.
\]
Thus \(|u|\) is a nonnegative subsolution of the Schrödinger inequality with potential
\(|V|\).

\textcolor{red}{Can one use the Nash-Moser iteration directly instead of reducing to nonnegative subsolution using Kato's inequality? we just multiply $\eta^2 |u|^{2m-2}\bar u$ in the later caccioppoli  inequality.}

If
\(
    0<r<R\le1
\)
and \(\eta\in C^\infty_0(B_R)\) satisfies \(\eta=1\) on \(B_r\) and
\(|\nabla\eta|\le C(R-r)^{-1}\), then for every \(m\ge1\),
\[
    \|\nabla(\eta |u|^m)\|_{L^2(B_R)}^2
    \le
    C m^2
    \left(
        (R-r)^{-2}
        +
        A^{\frac{2t}{2t-d}}
    \right)
    \||u|^m\|_{L^2(B_R)}^2.
\]

\textcolor{red}{ Does it seem to be $m^{\frac{2t}{2t-d}}$ instead of $m^2$ in the last inequality? If $t>d$, then ${\frac{2t}{2t-d}}<2$. If $\frac d 2 <t<d,$ then $ {\frac{2t}{2t-d}}>2$. However, it does not play an important role in the later convergence when using Moser's estimates. Let me show why it is $m^{\frac{2t}{2t-d}}$ below.}

Let $w=|u|, W=|V|, A=\|V\|_{L^t(B_1}$
Use the nonnegative test function
$\varphi=\eta^2 w^{2m-1},$

Then
\[
\int_\Omega \nabla w\cdot\nabla\!\left(\eta^2w^{2m-1}\right)
\le
\int_\Omega W\eta^2w^{2m}.
\]

Expanding the left-hand side,
\[
(2m-1)\int_\Omega
\eta^2w^{2m-2}|\nabla w|^2
+
2\int_\Omega
\eta w^{2m-1}\nabla w\cdot\nabla\eta
\le
\int_\Omega
W\eta^2w^{2m}.
\]

Estimating the middle term by Young's inequality,
\[
\left|
2\int_\Omega
\eta w^{2m-1}\nabla w\cdot\nabla\eta
\right|
\le
\frac{2m-1}{2}
\int_\Omega
\eta^2w^{2m-2}|\nabla w|^2
+
\frac{2}{2m-1}
\int_\Omega
w^{2m}|\nabla\eta|^2.
\]

Hence,
\[
(2m-1)\int_\Omega
\eta^2w^{2m-2}|\nabla w|^2
\le
\frac{C}{2m-1}
\int_\Omega
w^{2m}|\nabla\eta|^2
+
C
\int_\Omega
W\eta^2w^{2m}.
\]

Since
\[
|\nabla(w^m)|^2
=
m^2w^{2m-2}|\nabla w|^2,
\]
it follows that
\[
\int_\Omega
\eta^2|\nabla(w^m)|^2
\le
C 
\int_\Omega
w^{2m}|\nabla\eta|^2
+
C m
\int_\Omega
W\eta^2w^{2m}.
\]

Using
\[
|\nabla(\eta w^m)|^2
\le
2\eta^2|\nabla(w^m)|^2
+
2w^{2m}|\nabla\eta|^2,
\]
we obtain
\[
\|\nabla(\eta w^m)\|_{L^2(\Omega)}^2
\le
C 
\int_\Omega
w^{2m}|\nabla\eta|^2
+
C m
\int_\Omega
W\eta^2w^{2m}.
\]

If \(\eta\) satisfies
\[
|\nabla\eta|
\le
\frac{C}{R-r},
\]
then
\[
\|\nabla(\eta w^m)\|_{L^2(\Omega)}^2
\le
C (R-r)^{-2}
\|w^m\|_{L^2(B_R)}^2
+
C m
\int_{B_R}
W\eta^2w^{2m}.
\tag{11}
\]
Let $f=\eta w^m$.
\textcolor{red}{
\[
\int_{B_R} W f^2
    \le \|W\|_{L^t(B_R)}
       \|f^2\|_{L^{\frac{t}{t-1}}(B_R)}
    = A\|f\|_{L^{\frac{2t}{t-1}}(B_R)}^2,
\] }
where
$A=\|V\|_{L^t(B_R)}.$
Let
$2^*=\frac{2d}{d-2}.$
Since \(t>\frac d2\),
$2<\frac{2t}{t-1}<2^*.$

Interpolate between \(L^2\) and \(L^{2^*}\):
\[
\|f\|_{L^{\frac{2t}{t-1}}}
\le
\|f\|_{L^2}^{\,1-\theta}
\|f\|_{L^{2^*}}^{\,\theta},
\]
where \(\theta\) is determined by
\[
\frac{t-1}{2t}
=
\frac{1-\theta}{2}
+\frac{\theta}{2^*}.
\]

Since
\[
\frac1{2^*}
=
\frac12-\frac1d,
\]
we obtain
\[
\frac1{2t}
=
\frac{\theta}{d},
\]
hence
\[
\theta=\frac{d}{2t}.
\]

Therefore,
\[
\int_{B_R}Wf^2
\le
A
\|f\|_{L^2}^{\,2(1-\theta)}
\|f\|_{L^{2^*}}^{\,2\theta}.
\]

By Sobolev's inequality,
\[
\|f\|_{L^{2^*}}^2
\le
C\|\nabla f\|_{L^2}^2.
\]
Thus
\[
\int_{B_R}Wf^2
\le
CA
\|f\|_{L^2}^{\,2(1-\theta)}
\|\nabla f\|_{L^2}^{\,2\theta}.
\]

Young's inequality gives, for every \(\varepsilon>0\),
\[
A
\|f\|_{L^2}^{\,2(1-\theta)}
\|\nabla f\|_{L^2}^{\,2\theta}
\le
\varepsilon
\|\nabla f\|_{L^2}^2
+
C
\varepsilon^{-\frac{\theta}{1-\theta}}
A^{\frac1{1-\theta}}
\|f\|_{L^2}^2.
\]

Since
\[
\frac1{1-\theta}
=
\frac1{1-\frac{d}{2t}}
=
\frac{2t}{2t-d},
\]
and
\[
\frac{\theta}{1-\theta}
=
\frac{\frac{d}{2t}}
{1-\frac{d}{2t}}
=
\frac{d}{2t-d},
\]
we conclude that
\[
\int_{B_R}Wf^2
\le
\varepsilon
\|\nabla f\|_{L^2}^2
+
C
\varepsilon^{-\frac{d}{2t-d}}
A^{\frac{2t}{2t-d}}
\|f\|_{L^2}^2.
\tag{33}
\]

Substituting \((33)\) into \((11)\), and choosing \(\varepsilon>0\) small  (say so $\varepsilon Cm\leq \frac{1}{2}$ )that the gradient term can be absorbed into the left-hand side, yields
\[
\|\nabla(\eta w^m)\|_{L^2(B_R)}^2
\le
C m^\gamma
\left(
(R-r)^{-2}
+
A^{\frac{2t}{2t-d}}
\right)
\|w^m\|_{L^2(B_R)}^2,
\]
where $\gamma=\frac{2t}{2t-d}$

\textcolor{red}{The proof of the comment ends here.}

Indeed, the potential term is estimated by
\[
    \int_{B_R}|V|\eta^2 |u|^{2m}
    \le
    \|V\|_{L^t(B_R)}
    \|\eta |u|^m\|_{L^{2t/(t-1)}(B_R)}^2,
\]
and \(L^{2t/(t-1)}\) is interpolated between \(L^2\) and the Sobolev exponent. The
Sobolev term is then absorbed by Young's inequality, producing the power
\(2t/(2t-d)\).

Applying Sobolev's inequality to \(\eta |u|^m\) and iterating over a decreasing sequence
of balls gives the Moser estimate
\[
    \|u\|_{L^\infty(B_{1/2})}
    \le
    C
    \left(
        1+A^{\frac{2t}{2t-d}}
    \right)^{d/4}
    \|u\|_{L^2(B_1)}.
\]
Scaling back to \(B_R(x_0)\) gives
\[
    \|u\|_{L^\infty(B_{R/2}(x_0))}
    \le
    C R^{-d/2}
    \left(
        1+
        \left(
            R^{2-d/t}\|V\|_{L^t(B_R(x_0))}
        \right)^{\frac{2t}{2t-d}}
    \right)^{d/4}
    \|u\|_{L^2(B_R(x_0))}.
\]
 This proves the lemma.
\end{proof}

We now prove the normalized oscillation estimate needed in Section~\ref{s:Remez}.

\begin{lemma}
\label{lem:complex-oscillation}
Let \(d\ge3\), \(d/2<t\le\infty\), and let
\(u\in W^{1,2}_{\mathrm{loc}}(B_2;\mathbb C)\) be a weak solution of
\[
    \Delta u+V(x)u=0
    \qquad\text{in }B_2,
\]
where \(V\in L^t(B_2;\mathbb C)\). Define
\(
    \alpha=\frac12\min\{1,2-d/t\}\) when \(t<\infty\) and \(\alpha=\frac{1}{2}\) when $t=\infty$.
Then for every \(0<r\le 1\),
\[
    \operatorname{osc}^{\mathbb C}_{B_r}u
    \le
    C r^\alpha
    \left(1+\|V\|_{L^t(B_2)}\right)
    \|u\|_{L^\infty(B_2)},
\]
where \(C=C(d,t)\). 
\end{lemma}

\begin{proof}
Assume first that \(t<\infty\). By Lemma~\ref{lem:complex-local-regularity},
\(u\in L^\infty_{\mathrm{loc}}(B_2)\) and \(u\in W^{2,t}_{\mathrm{loc}}(B_2)\).
Since \(\Delta u=-Vu\), the interior \(W^{2,t}\)-estimate for the Laplacian gives
\[
    \|u\|_{W^{2,t}(B_1)}
    \le
    C
    \left(
        \|\Delta u\|_{L^t(B_2)}
        +
        \|u\|_{L^t(B_2)}
    \right).
\]
Hence
\[
    \|u\|_{W^{2,t}(B_1)}
    \le
    C
    \left(
        \|Vu\|_{L^t(B_2)}
        +
        \|u\|_{L^t(B_2)}
    \right)
    \le
    C
    \left(1+\|V\|_{L^t(B_2)}\right)
    \|u\|_{L^\infty(B_2)}.
\]
By Morrey's embedding, \(W^{2,t}(B_1)\hookrightarrow C^\alpha(B_1)\) for the above
choice of \(\alpha\). Therefore, for all \(y,z\in B_1\),
\[
    |u(y)-u(z)|
    \le
    C |y-z|^\alpha
    \left(1+\|V\|_{L^t(B_2)}\right)
    \|u\|_{L^\infty(B_2)}.
\]
Taking the supremum over \(y,z\in B_r\), with \(0<r\le1\), gives
\[
    \operatorname{osc}^{\mathbb C}_{B_r}u
    \le
    C r^\alpha
    \left(1+\|V\|_{L^t(B_2)}\right)
    \|u\|_{L^\infty(B_2)}.
\]

If \(t=\infty\), choose a fixed finite exponent \(p>d\), for example \(p=2d\). Then
\[
    \|V\|_{L^p(B_2)}
    \le
    C\|V\|_{L^\infty(B_2)}.
\]
Applying the previous argument with \(p=2d\) in place of \(t\) gives
\[
    \operatorname{osc}^{\mathbb C}_{B_r}u
    \le
    C r^{1/2}
    \left(1+\|V\|_{L^\infty(B_2)}\right)
    \|u\|_{L^\infty(B_2)}.
\]
This completes the proof.
\end{proof}

\begin{remark}
The same estimate holds after scaling. If $\Delta u+Vu=0$ in \(B_{2R}(x_0)\), then for
\(0<r\le R\),
\[
    \operatorname{osc}^{\mathbb C}_{B_r(x_0)}u
    \le
    C
    \left(\frac rR\right)^\alpha
    \left(
        1+R^{2-d/t}\|V\|_{L^t(B_{2R}(x_0))}
    \right)
    \|u\|_{L^\infty(B_{2R}(x_0))}.
\]
For \(t=\infty\), the factor \(R^{2-d/t}\|V\|_{L^t}\) is understood as
\(R^2\|V\|_{L^\infty}\).
\end{remark}
\fi

\bibliography{Mybib}
\bibliographystyle{abbrv}
\end{document}